\documentclass{article}

\usepackage[utf8]{inputenc}
\usepackage[T1]{fontenc}
\usepackage[english]{babel}
\usepackage{mathpazo} 
\usepackage{sansmath} 

\usepackage[left=2.3cm, right=2.3cm, top=3cm, bottom=3cm]{geometry}
\usepackage[small, sc]{titlesec}
\usepackage{placeins}
\usepackage{environ}
\usepackage{verbatimbox}
\usepackage{longtable}
\usepackage{caption}
\usepackage{subcaption}
\usepackage[active]{srcltx} 

\usepackage{amsmath}
\usepackage{amssymb}
\usepackage{amsthm}      
\usepackage{mathtools}
\usepackage{empheq}
\usepackage{bigints}
\usepackage{mathrsfs}
\usepackage{cancel}
\usepackage{bm}          
\usepackage{xfrac}

\usepackage{graphicx}
\usepackage{xcolor}
\usepackage{colorspace}
\usepackage{tikz}
\usepackage{pgfplots}
\usepackage{tkz-euclide}

\pgfplotsset{compat=1.10}
\usepgfplotslibrary{fillbetween}
\usetikzlibrary{patterns, intersections, calc, angles, arrows, automata, positioning, through, shadings}
\pgfdeclarelayer{pre main}

\makeatletter
\newsavebox{\measure@tikzpicture}
\NewEnviron{scaletikzpicturetowidth}[1]{%
	\def\tikz@width{#1}%
	\begin{lrbox}{\measure@tikzpicture}%
		\BODY
	\end{lrbox}%
	\pgfmathparse{#1/\wd\measure@tikzpicture}%
	\BODY
}
\makeatother

\tikzset{
	state/.style={
		rectangle,
		rounded corners,
		draw=black, very thick,
		minimum height=2em,
		inner sep=2pt,
		text centered,
	},
}

\usepackage{aliascnt}

\theoremstyle{plain}
\newtheorem{theorem}{Theorem}[section]

\newaliascnt{lemma}{theorem}
\newtheorem{lemma}[lemma]{Lemma}
\aliascntresetthe{lemma}

\newaliascnt{proposition}{theorem}
\newtheorem{proposition}[proposition]{Proposition}
\aliascntresetthe{proposition}

\newaliascnt{corollary}{theorem}
\newtheorem{corollary}[corollary]{Corollary}
\aliascntresetthe{corollary}

\newaliascnt{conjecture}{theorem}

\aliascntresetthe{conjecture}

\newaliascnt{axiom}{theorem}

\aliascntresetthe{axiom}

\theoremstyle{definition}
\newaliascnt{definition}{theorem}

\aliascntresetthe{definition}

\newtheorem*{notation*}{Notation}

\newtheorem*{theorem*}{Theorem}
\newtheorem*{corollary*}{Corollary}

\theoremstyle{remark}
\newaliascnt{remark}{theorem}
\newtheorem{remark}[remark]{Remark}
\aliascntresetthe{remark}

\numberwithin{equation}{section}
\numberwithin{figure}{section}

\usepackage[
backend=bibtex,     
style=numeric,
sorting=nty,
maxbibnames=50,
giveninits=true
]{biblatex}
\usepackage{csquotes}

\apptocmd{\thebibliography}{\fontsize{11}{15}\selectfont}{}{}%
\renewbibmacro{in:}{\ifentrytype{article}{}{}}
\DeclareFieldFormat[article]{volume}{\textbf{#1}}
\DeclareFieldFormat[article]{title}{\textit{#1}}
\DeclareFieldFormat[article]{journaltitle}{#1}
\DeclareFieldFormat[article]{number}{\textbf{#1}}
\renewbibmacro*{volume+number+eid}{%
	\printfield{volume}%
	\setunit*{\textbf{\addcolon}}%
	\printfield{number}\setunit{\addcomma\space}%
	\printfield{eid}}

\newcommand{\E}{\mathcal{E}}
\newcommand{\F}{\mathcal F}

\newcommand{\R}{\mathbb{R}}

\newcommand{\vect}[1]{\boldsymbol{#1}}
\newcommand{\tens}[1]{\mathsf{#1}}
\newcommand{\abs}[1]{\left\lvert#1\right\rvert}
\newcommand{\n}{\mathbf{n}}
\DeclareMathOperator{\diver}{div}

\newcommand{\adjustedaccent}[1]{%
	\mathchoice{}{}
	{\mbox{\raisebox{-.5ex}[0pt][0pt]{$\scriptstyle#1$}}}
	{\mbox{\raisebox{-.35ex}[0pt][0pt]{$\scriptscriptstyle#1$}}}
}
\newcommand\smileacc[1]{\overset{\adjustedaccent{\smallsmile}}{#1}}
\newcommand\frownacc[1]{\overset{\adjustedaccent{\smallfrown}}{#1}}

\definecolor{myred}{rgb}{0.9,0,0}

\definecolor{vargreen}{rgb}{0.0, 0.5, 0.0}

\newcommand{\at}[3][]{#1|_{#2}^{#3}}

\renewcommand{\rho}{\varrho}
\renewcommand{\theta}{\vartheta}

\usepackage{hyperref} 
\hypersetup{colorlinks=true, linkcolor=blue, citecolor=red}

\usepackage[capitalise, nameinlink, noabbrev]{cleveref}

\crefname{lemma}{lemma}{lemmas}
\Crefname{lemma}{Lemma}{Lemmas}

\crefname{proposition}{proposition}{propositions}
\Crefname{proposition}{Proposition}{Propositions}

\crefname{corollary}{corollary}{corollaries}
\Crefname{corollary}{Corollary}{Corollaries}

\crefname{definition}{definition}{definitions}
\Crefname{definition}{Definition}{Definitions}

\crefname{remark}{remark}{remarks}
\Crefname{remark}{Remark}{Remarks}

\crefname{axiom}{axiom}{axioms}
\Crefname{axiom}{Axiom}{Axioms}

\crefname{conjecture}{conjecture}{conjectures}
\Crefname{conjecture}{Conjecture}{Conjectures}

\crefname{theorem}{theorem}{theorems}
\Crefname{theorem}{Theorem}{Theorems}
\begin{document}
\title{\textsc{Existence and uniqueness of minimizers\\ for axisymmetric nematic films}}
\author{\textsc{G.\ Bevilacqua}$^1$\thanks{\href{mailto:giulia.bevilacqua@dm.unipi.it}{\texttt{giulia.bevilacqua@dm.unipi.it}}}\,\,\,$-$\,\, \textsc{C.\ Lonati}$^2$\thanks{\href{mailto:chiara.lonati@polito.it}{\texttt{chiara.lonati@polito.it}}}\,\,\,$-$\,\, \textsc{L.\ Lussardi}$^2$\thanks{\href{mailto:luca.lussardi@polito.it}{
\texttt{luca.lussardi@polito.it}}}\,\,\,$-$\,\,\textsc{A.\  Marzocchi}$^3$\thanks{\href{mailto:alfredo.marzocchi@unicatt.it}{
\texttt{alfredo.marzocchi@unicatt.it}}}
\bigskip\\
\normalsize$^1$ Dipartimento di Matematica, Università di Pisa, Largo Bruno Pontecorvo 5, I–56127 Pisa, Italy.\\
\normalsize$^2$  DISMA, Politecnico di Torino, c.so Duca degli Abruzzi 24, I-10129 Torino, Italy.\\
\normalsize$^3$ Dipartimento di Matematica e Fisica ``N. Tartaglia", Università Cattolica del Sacro Cuore,\\
\normalsize via della Garzetta 48, I-25133 Brescia, Italy\\
}
\date{}

\maketitle

\begin{abstract}
\noindent Nematic surfaces are thin liquid films endowed with in-plane orientational order.
We study a variational model in which the nematic director is constrained to lie in the tangent space of an axisymmetric surface, and the associated surface energy accounts for both surface tension and elastic nematic contributions.
Here we adopt the surface gradient as the differential operator on the surface, we restrict our analysis to revolution surfaces spanning two coaxial rings, and we assume that the nematic director is aligned along parallels. In this setting, the energy functional reduces to a one-dimensional variational problem. We rigorously prove the existence and uniqueness of minimizers and we provide their complete geometric characterization.  Finally, we run some numerical simulations.
\end{abstract}

\bigskip
\bigskip

\textbf{Mathematics Subject Classification (2020)}: 49J05, 49Q10, 49J45, 76A15.

\textbf{Keywords}: thin films, nematic surfaces, one-dimensional variational problems.

\bigskip

\tableofcontents

\section{Introduction}

Liquid crystals are peculiar materials that share properties of both liquids and solids. 
Since their discovery in 1888 \cite{lehmann1889fliessende}, they have been found in many natural systems: 
proteins and cell membranes are examples of liquid crystalline structures \cite{saw2017topological}. 
More recently, synthetic versions, like for instance polymeric liquid crystals, have attracted great attention thanks to their technological applications, such as display devices \cite{wissbrun2006orientation} and optical switches \cite{goodby1989eutectic}. 

A useful mathematical approach to model these phenomena is the \emph{Oseen--Frank theory} which introduces an elastic energy required to deform a liquid crystal from a uniform ground state to, for example, a nematic one \cite{Virga1994}. Such a theory has been widely employed by characterizing equilibrium configurations on different geometries \cite{Biscari_2006,RossoVirga2011, LopezLeon2011FrustratedNO, Napoli-turzi2020, Segatti-veneroni2014}, by considering different boundary conditions \cite{lonberg1985new, napoli2005weak} and as static limit \cite{yue2004diffuse, cates2018theories, bevilacqua2023global} of nematic flows introduced by Ericksen and Leslie \cite{ericksen1976equilibrium}.

Precisely, the Oseen-Frank theory is mainly adopted to study the most common phase of a liquid crystal, i.e. the {\em nematic phase}. It is made up of rod-like molecules that are on average arranged following an orientational order: their longest axes are approximately pointing in a direction specified by a unit vector called {\em director} and usually denoted by $\n$. In such a context, $\n$ and $-\n$ are equivalent because of head-tail symmetry due to lack of polarity. The most general quadratic form of the energy is given by
$$
W(x, \n, \nabla \n) = k_1 \abs{\diver \n}^2 + k_2 \left(q + \n\cdot {\rm curl}\,\n\right)^2 + k_3 \abs{\n \times {\rm curl}\,\n}^2 + \left(k_4+k_5\right)\left[ {\rm tr}\, (\nabla \n)^2 - (\diver \n)^2\right],
$$
where $k_i$ and $q$ are material constants, typically depending on the temperature. 
Considering the energy functional $\int_\Omega W(x, \n, \nabla \n)\, dx$ on a regular domain $\Omega \subset \mathbb{R}^3$, 
the theory of existence and characterization of minimizers is rather well understood \cite{hardt1986existence, lin1989nonlinear}.  

However, much less is known in the case where the nematic director $\n$ is constrained to lie on a \emph{free surface}.
Here the situation becomes particularly rich: the system must balance the area of the surface, penalized by surface tension, 
with the nematic energy, which tends to favor flat configurations with uniform orientation. 
This creates an interesting coupling: as the surface deforms, it influences the orientation of $\n$, 
and conversely, the orientation of $\n$ modifies the optimal shape of the surface. 
Understanding this interplay is not only mathematically challenging but also crucial for future applications, 
such as flexible screens and adaptive lenses.  

We consider the {\em one-constant approximation} of the Oseen-Frank theory for liquid crystals \cite{napoli2012surface}, i.e. $k_1 = k_2= k_3 = \kappa$, $k_4 = k_5 = 0$ and $q = 0$ and we require that the unit nematic vector $\n$ is constrained to lie in the tangent space of the surface $S$. Precisely, the energy functional we aim to study is given by
\begin{equation}
\label{eq:funzionale_generale}
    \mathcal{E}(\n, S)=\int_S \left(\gamma+\frac{\kappa}{2}|\nabla \n|^2\right)dA,
\end{equation}
where $S$ is a smooth surface with a prescribed boundary, $\gamma$ is the constant surface tension, $\kappa$ is the nematic constant and $\nabla \n$ is a gradient of $\n$ (the choice of the differential operator on the surface will be more precise later). Notice that when $\kappa=0$, the problem \eqref{eq:funzionale_generale} reduces to the classical Plateau problem (see \cite{david2014chapter} and for some recent developments \cite{giusteri2017solution, bevilacqua2019soap, bevilacqua2020dimensional, palmer2021minimal}).

Our first attempt in studying \eqref{eq:funzionale_generale} was done in \cite{bllm}: we
chose as differential operator on $S$ the covariant derivative $\nabla = \tens D$.
Moreover, we rewrote $\mathcal{E}(\n, S)$ for revolution surfaces $S$ parametrized by the profile curve $\rho:[-h,h] \to (0, +\infty)$ and constrained to span two coaxial rings of radius $r$ placed at distance $2h$, namely $\rho(\pm h) = r$. Thus, since in the minimization process it is convenient to choose a constant angle $\alpha$, which is the one formed by $\n$ with the parallels, $\mathcal{E}(\n,S)$ reduced to a purely geometric energy functional of the form
$$
\mathcal{E}^{\rm cov}(\rho) = \bigintsss_{-h}^h\left(\rho\sqrt{1+(\rho')^2}+ c\,\frac{(\rho')^2}{\rho\sqrt{1+(\rho')^2}} \right)\,dx,
$$
where the superscript ``cov'' refers to the choice of the differential operator (covariant derivative) on $S$ and the parameter
$$
c := \frac{\kappa}{2 \gamma}
$$
affects the shape of solutions. In \cite[Theorem 1.1]{bllm}, we rigorously proved that for $\sfrac{h}{r}$ sufficiently small and for all $c \geq 0$, the energy functional $\mathcal{E}^{\rm cov}(\rho)$ admits at least a minimizer.

However, the covariant derivative
neglects the extrinsic curvature of 
the surface $S$.
For this reason, in \cite{napoli2012surface, napoli2018influence}, Napoli and Vergori studied the energy functional $\mathcal{E}(\n, S)$ when $\nabla=\nabla_S$, namely they considered  the surface gradient defined as follows (for more details see \cite{gurtin1975continuum}). Let $x \in S$ and let $\pi_x \colon \mathcal N_x\to S$ be the projection on $S$, where $\mathcal N_x$ is a neighborhood of $0$ on $T_xS$. Thus, $\nabla_S\n(x) := \nabla (\n \circ \pi_x)(0)$ and 
$$
|\nabla_S\n|^2 = \abs{\tens{D}\n}^2 + \abs{\tens{L}\n}^2,
$$
where $\tens{L}$ is the extrinsic curvature tensor of the surface $S$. More precisely, in \cite{napoli2018influence} the case of revolution surfaces spanning two equal coaxial circles of radius $r$ placed at distance $2 h$ is considered. The Authors derived the Euler-Lagrange equations and they studied the equilibrium configurations in terms of the parameter $c$, see for instance Figures 3 and 4 of \cite{napoli2018influence}.

Therefore, motivated by \cite{napoli2018influence}, in this paper, we prove that their equilibrium solutions are actually minimizers as well as we rigorously justify the mentioned pictures and some of their numerical results.

In \Cref{app:derivazione}, we rewrite $\E(\n, S)$ with $\nabla = \nabla_S$ for our specific case of revolution surfaces. Similarly as done in \cite{bllm}, $S$ is parametrized by the profile curve $\rho\colon [-h,h]\to (0,+\infty)$ subjected to the constraint $\rho(\pm h)=r$, and the nematic director is defined as $\n=\cos \alpha \vect e_1+\sin \alpha \vect e_2$, where $\vect e_1$ is the parallel direction, while $\vect e_2$ is the meridian one. Here, as done in \cite{napoli2018influence}, we focus on the case $\alpha = 0$ allowing the nematic director $\n$ to be aligned with the parallels. 
Thus, again in this context, the minimization problem of $\E(\n, S)$ reduces to the minimization of a purely geometric energy functional given by
\begin{equation*}
    \mathcal{E}(\rho)=\bigintsss_{-h}^h\left(\rho + \frac{c}{\rho}\right)\sqrt{1+(\rho')^2} \,dx,
\end{equation*}
where $c$ is as before.
Similarly as for the covariant derivative case \cite{bllm}, even if the energy functional does not depend on $\alpha$, it is reminiscent of the presence of the nematic director through the term
$$
\frac{\sqrt{1+(\rho')^2}}{\rho}.
$$

In the next section, we set the minimization problem in $W^{1,1}$ and we state our main result.

\subsection{Setting of the problem and main result}

Let $h,r>0$ and let 
$$X:=\left\{\rho \in W^{1,1}(-h,h) : \rho> 0,\,\rho(-h)=\rho(h)=r\right\}.$$
Let $c\ge 0$ and let $\mathcal{F}_c \colon X \to \mathbb R$ be the energy functional given by 
\begin{equation*}
     \mathcal{F}_c(\rho)=\bigintsss_{-h}^hf(\rho)\sqrt{1+(\rho')^2}\,dx,\ 
\end{equation*}
where 
\[
f(s)=s+\frac{c}{s}.
\]
Let $\Xi$ be the unique positive solution of the equation 
\[
\Xi \tanh \frac{1}{\Xi}+\text{sech}^2\frac{1}{\Xi}=\Xi 
\]
and let 
\begin{equation*}
    \omega=\frac{1}{\Xi\cosh(\sfrac{1}{\Xi})}. 
\end{equation*}
It turns out (for details, see \cite[Theorem 3.3]{bllm}) that if 
\begin{equation}\label{cond_catenaria}
\frac{h}{r}\le \omega,
\end{equation}
then the equation 
\[
r=\Pi \cosh\frac{h}{\Pi}
\]
has two distinct strictly positive solutions $\Pi_0(h,r)>\Pi_1(h,r)$. We remark that the function 
\begin{equation*}
 \left\{
 \begin{aligned}
    [-h,h]&\to \R\\
 x &\mapsto\rho_0(x)=\Pi_0\cosh\frac{x}{\Pi_0}
 \end{aligned}
 \right.
\end{equation*}
belongs to $X$ and it is the unique minimizer of $\F_0$ (the classical catenary problem). 

Throughout this paper, $h,r$ will always be fixed and they will always satisfy \eqref{cond_catenaria}. Moreover, we underline that also the parameter $c>0$ is fixed. 

We are in position to state our main result.

\begin{theorem}\label{thm_existence}
The functional $\mathcal F_c$ admits a unique minimizer $\rho_c \in X$. Moreover, the following complete characterization of the minimizer $\rho_c$ holds true.
\begin{itemize}
\item[a)] If $c<r^2$, in the equation 
\begin{equation*}
\cosh\frac{2h}{E}=\frac{2r^2-E^2+2c}{E\sqrt{E^2-4c}},
\end{equation*}
there exists only one solution $\smileacc{E}(c)>\Pi_0(h,r)$ and 
\[
\rho_c(x)=\sqrt{
    \frac{\smileacc{E}}{2}\sqrt{\smileacc{E}^2-4c}\cosh\frac{2x}{\smileacc{E}}+\frac{\smileacc{E}^2}{2}-c}.
\]
In this case, $\rho_c\in C^\infty([-h,h])$, it is even, strictly convex and $\rho_0<\rho_c<r$ on $(-h,h)$.
\item[b)] If $c=r^2$, then $\rho_c(x)=r$ for all $x \in [-h,h]$.
\item[c)] If $c>r^2$, the equation 
\begin{equation*}
\cosh\frac{2h}{E}=\frac{-2r^2+E^2-4c}{E\sqrt{E^2-4c}}
\end{equation*}
has a unique solution $\frownacc{E}(c)$ which is positive and 
\[
\rho_c(x)=\sqrt{
    -\frac{\frownacc{E}}{2}\sqrt{\frownacc{E}^2-4c}\cosh\frac{2x}{\frownacc{E}}+\frac{\frownacc{E}^2}{2}-c}.
\]
In this case, $\rho_c\in C^\infty([-h,h])$, it is even, strictly concave and $r<\rho_c<\rho_\infty$ on $(-h,h)$ where 
$$\rho_\infty(x)=\sqrt{h^2+r^2-x^2}.$$
Moreover, $\rho_c\to \rho_\infty$ uniformly on $[-h,h]$ as $c\to +\infty$. 
\end{itemize}
Finally, $\rho_{c_1}<\rho_{c_2}$ on $(-h,h)$ whenever $c_1<c_2$.
\end{theorem}

    The proof is a combination of the following steps obtained in the next three sections:
    \begin{itemize}
\item in \Cref{sec:esistenza_minimo_rhoc}, for all $c>0$, we prove the existence of a minimizer $\rho_c\in X$ for $\mathcal{F}_c$;
\item in \Cref{sec:studio_bvp}, we show that above $\rho_0$, the unique stable catenary for $c = 0$, the minimizer $\rho_c$ is unique; we provide the monotonicity property with respect to the parameter $c$ and we characterize the shape of the minimizer $\rho_c$ solving a boundary value problem;
\item in \Cref{sec:c_tinfty}, by essentially a $\Gamma$-convergence procedure, we characterize the behavior of minimizers as $c \to \infty$.
\end{itemize}
Finally, in \Cref{sec:numerica} we run some numerical simulations.

\subsection{Strategy of the proof}

Similarly as done in \cite{bllm}, we decide to minimize the energy functional $\mathcal{F}_c$ in $W^{1,1}$ and not in $BV$ and this requires a more delicate analysis.
Differently from \cite{bllm}, it is not true a posteriori that the energy decreases by convexification: numerical simulations performed in \cite{napoli2018influence} (see their Figure 3) suggest that the shape of solutions strongly depends on the regime of the parameter $c$. For instance, in the case $c=r^2$, it is easy to see that the constant function $\rho=\sqrt c$ is the unique minimizer. If \eqref{cond_catenaria} holds, namely we are in the regime where the catenary $\rho_0$ is the unique minimizer for $\mathcal F_0$, if $c < r^2$, then the profile is convex while it is concave for $c>r^2$. For this reason, we proceed in a different way to show \Cref{thm_existence}.

Since the energy density of $\mathcal {F}_c$ is not coercive and it has linear growth in $\rho'$, to look for minimizers of $\mathcal{F}_c$ over $X$ we follow \cite[Theorem 5.1]{bm1991}. In \Cref{sec:relaxation} we relax $\F_c$ weakly in $W^{1,1}_{\rm loc}(-h,h)$. The relaxed functional is defined on the weak closure of $X$, which is $W^{1,1}_{\rm loc}(-h,h)$. The main problem in minimizing $\F_c$ on such a set realizes on the fact that a generic $\rho\in W^{1,1}_{\rm loc}(-h,h)$ may vanish at some point. In order to prevent this, we need to distinguish two cases. 
\begin{itemize}
    \item If $c\in (0,r^2)$, in \Cref{sec:c<r2}, we show that it is energetically convenient to lie above $\sqrt c$: we are able to prove that $\F_c(\rho \vee \sqrt c)\le \F_c(\rho)$. This allows (see \Cref{thm:minimo_rilassato}) to minimize the relaxed energy functional on 
$$
\left\{\rho\in W^{1,1}_{\rm loc}(-h,h):\rho \geq \sqrt c\right\}.
$$
Using a first variation argument, we prove that minimizers obtained by this procedure are convex and they satisfy the boundary conditions $\rho(\pm h)=r$ in the usual Sobolev sense. Finally, by \Cref{tilli1} in the regime where \eqref{cond_catenaria} holds true, we show that $\F_c(\rho \vee \rho_0)\le \F_c(\rho)$, implying that minimizers belong to $W^{1,1}(-h,h)$. 
    \item If $c>r^2$, in \Cref{sec:c>r^2}, by  \Cref{thm:minimo_rilassato}, we minimize the relaxed energy functional on 
$$
\left\{\rho\in W^{1,1}_{\rm loc}(-h,h):\rho \geq c_0\right\},
$$
where $c_0=\min \rho_0$. Using again a first variation argument, we prove that minimizers obtained by this procedure are concave and they satisfy the boundary conditions $\rho(\pm h)=r$ in the classical Sobolev sense. In order to conclude, we need to show that it is convenient to have a suitable bound from above. Precisely, by \Cref{tilli2} we show that $\F_c(\rho \wedge \rho_\infty)\le \F_c(\rho)$, where $\rho_\infty$ is the unique minimizer (obtained in \Cref{lemma:esiste_minimo_Finfty}) of 
\[
\mathcal{F}_\infty(\rho) = \bigintsss_{-h}^h \frac{\sqrt{1+ (\rho')^2}}{\rho}\, dx,
\]
and this yields the conclusion.
\end{itemize}

By deriving the Euler-Lagrange equation, in \Cref{sec:studio_bvp}, we get that any minimizer of $\F_c$ satisfies the boundary value problem 
\[
\left\{
\begin{aligned}
    &\rho''=\frac{(\rho^2-c)(1+(\rho')^2)}{\rho(\rho^2+c)}\\
    &\rho(-h) = r\\
&\rho(h)=r.
\end{aligned}
\right.
\]
In \Cref{prop:proprieta_bvp} we get the explicit expressions of solutions:
$$
\rho_c(x)=\sqrt{ \frac{E}{2}\sqrt{E^2-4c}\cosh\left(\frac{2x}{E}\right)+\frac{E^2}{2}-c}, \qquad \hbox{if } c<r^2,
$$
and
$$
\rho_c(x)=\sqrt{-\frac{E}{2}\sqrt{E^2-4c}\cosh\left(\frac{2x}{E}\right)+\frac{E^2}{2}-c}, \qquad \hbox{if }  c>r^2,
$$
where $E=E(c)>0$ is a solution of a transcendental equation coming from the imposition of boundary conditions for $\rho_c$. We point out  that these solutions were already found by Napoli and Vergori \cite{napoli2012surface}. Due to the fact that the transcendental equation has multiple solutions (in general), we do not have uniqueness of critical points. Nevertheless, in both regimes $c<r^2$ \Cref{lemma:unicita_soluzioni_1} and $c>r^2$ in \Cref{lemma:unicita_soluzioni_2}, we prove that there exists a unique critical point lying above the stable catenary $\rho_0$, and this gives the uniqueness of the minimizer $\rho_c$. 

Finally, in \Cref{sec:c_tinfty}, we remark that $c\mapsto \rho_c$ is an increasing function and $\rho_c\to \rho_\infty$ as $c\to +\infty$ uniformly on $[-h,h]$; this is reasonable since we expect that $\sfrac{\F_c}{c}$ $\Gamma$-converges to $\F_\infty$.

\section{Existence of minimizers}
\label{sec:esistenza_minimo_rhoc}

In this section we show the following proposition, which proves exactly the first part of \Cref{thm_existence}.
\begin{proposition}
\label{prop:esistenza_minimo_Ec}
    The energy functional $\mathcal{F}_c$ admits a minimizer $\rho_c \in C^2([-h,h])$.
\end{proposition}
\begin{remark}
\label{rem:c=r2}
    We remark that if $c = r^2$, then the constant function $\rho_c(x) = \sqrt{c}$ is the unique minimizer for $\mathcal{F}_c$. Indeed, since $\sqrt{1+ (\rho')^2}\geq 1$ and since $\min f=f(\sqrt c)$, we directly get 
    $$\mathcal{F}_c(\rho) \geq \mathcal{F}_c(\sqrt{c}).$$
    Moreover $\rho_c(\pm h) = \sqrt{c} = r$, namely boundary conditions are satisfied. Finally, if $\rho_c(\pm h) = \sqrt{c} $ and $\F_c(\rho)=\F_c(\sqrt c)$ we get 
    \[
    4h\sqrt c=\F_c(\sqrt c)=\bigintsss_{-h}^h f(\rho)\sqrt{1+ (\rho')^2}\,dx\ge 2\sqrt c \bigintsss_{-h}^h \sqrt{1+ (\rho')^2}\,dx.
    \]
    As a consequence, 
    \[
    \bigintsss_{-h}^h\sqrt{1+ (\rho')^2}\,dx\le 2h
    \]
    which means that $\rho'=0$ a.e.\,on $(-h,h)$, implying that $\rho=\sqrt c$.
\end{remark}
By \Cref{rem:c=r2}, from now on we can assume either $0<c<r^2$ or $c>r^2$.
To prove \Cref{prop:esistenza_minimo_Ec} we need some auxiliary results.

\subsection{Relaxation of non coercive energy functionals}
\label{sec:relaxation}
Let $a,b \in \R$ with $a<b$ and let $g\colon \R \to [0,+\infty)$ be a continuous function. Let us consider energy functionals of the form
    \[
    \label{e:tipico_funzionale}
    \mathcal{G}(\rho)= \bigintsss_a^b g(\rho)\sqrt{1+ (\rho')^2}\, dx,
\]
defined on
\[
   \label{e:tipico_insieme}
    \mathcal W:= \left\{\rho \in W^{1,1}(a,b):\ \rho(a) =c_1,\, \rho(b) = c_2\right\},
\]
for some 
$c_1,c_2\in \R$. 
Let $\overline{\mathcal W}\subseteq W^{1,1}_{\rm loc}(a,b)$ be the closure of $\mathcal W$ in $W^{1,1}_{\rm loc}(a,b)$ with respect to the weak topology of $W^{1,1}_{\rm loc}(a,b)$. For any $\rho \in \overline{\mathcal W}$, let 
\[
\overline{\mathcal G}(\rho)=\inf\left\{\liminf_{j\to+\infty}\mathcal G(\rho_j) : \rho_j\rightharpoonup \rho \,\text{in}\, W^{1,1}_{\rm loc},\,\rho_j\in \mathcal W\right\}.
\]
The following representation theorem holds true.
\begin{theorem}[{\cite[Theorem 2.1]{bm1991}}]\label{thm:rilassato}
We have $\overline{\mathcal W}=W^{1,1}_{\rm loc}(a,b)$. Moreover, the functional $\overline{\mathcal G}$ is given by
\begin{equation*}
    \overline{\mathcal{G}}(\rho)=\mathcal{G}(\rho)
    +\left|G(\rho(a))-G(c_1)\right|
    +\left|G(\rho(b))-G(c_2)\right|,
\end{equation*}
where $G$ is a primitive of $g$. Here, the traces of $\rho$ are defined in a weak sense, namely
\[
\rho(a)=\inf\left\{\liminf_{j\to+\infty}\rho(x_j) : x_j\to a\right\}, \quad \rho(b)=\inf\left\{\liminf_{j\to+\infty}\rho(x_j) : x_j\to b\right\}.
\]
\end{theorem}
The existence of a minimizer for the relaxed energy functional $\overline{\mathcal{G}}$ is guaranteed by the next theorem.
\begin{theorem}[\protect{\cite[Theorem 5.1]{bm1991}}]\label{thm:minimo_rilassato}
    \noindent Assume that $g\in C^2([a_0, + \infty))$, for some $a_0>0$, satisfies one of the following conditions: 
    \begin{itemize}
    \item[i)] $g>0$ is monotone;
    \item[ii)] if there exists $\overline{x} \in [a_0,+\infty)$ such that $g'(\overline{x}) = 0$ then $g''(\overline{x})>0$.
    \end{itemize}
    Then the problem
    $$
    \min\left\{\overline{\mathcal{G}}(\rho): \rho\in W_{\rm loc}^{1,1}(a,b),\,\rho\geq a_0\right\}
    $$
    admits a solution.
\end{theorem}

\subsection{Some preliminary estimates}

We first investigate an auxiliary functional related, in a suitable sense, to the case $c\to+\infty$. Let 
\begin{equation}
    \label{e:F_infty}
     \mathcal{F}_\infty(\rho):\left\{
    \begin{aligned}
        X &\to \R\\
        \rho&\mapsto\bigintsss_{-h}^h\frac{1}{\rho}\sqrt{1+(\rho')^2}\,dx.
    \end{aligned}
    \right.
\end{equation}

\begin{lemma}
\label{lemma:esiste_minimo_Finfty}
    The energy functional $\mathcal{F}_\infty$ defined in \eqref{e:F_infty} admits a unique minimizer which is given by 
    \[
    \rho_\infty(x)=\sqrt{h^2+r^2-x^2}.
    \]
\end{lemma}
\begin{proof}
    We divide the proof into some steps.\\
    \\
    {\em Step 1.} We claim that the relaxed energy functional $\overline{\mathcal{F}}_\infty$ admits a minimizer on 
    \[
    X_r=\left\{\rho \in W^{1,1}_{\rm loc}(-h,h) : \rho\ge r\right\}.
    \]
    It is sufficient to apply \Cref{thm:minimo_rilassato} with the choice $g \colon [r,+\infty) \to \R$ given by $g(s)=\sfrac{1}{s}$. 
\\
\\
{\em Step 2.} Let $\rho_\infty \in X_r$ be a minimizer for $\overline{\mathcal{F}}_\infty$. We claim that $\rho_\infty\in C^2(-h,h)$ is strictly concave and $\rho_\infty>r$ on $(-h,h)$. Let $x_0 \in (-h,h)$ be a point such that $\rho_\infty > r$. By continuity, there exists an interval $(a,b) \ni x_0$ such that $\rho(x) >r$ for all $x \in (a,b)$.
Let $\phi\in C^1_{c}(a,b)$. For any $\sigma >0$ small enough, we have that $\rho_\infty+t\phi \in X_r$ for each $t \in (- \sigma, \sigma)$. Then, we easily obtain 
\[
0=\frac{d}{dt}_{\big|_{t=0}} \overline{\mathcal{F}}_\infty(\rho_\infty + t\phi) =\bigintsss_{a}^b \left(\frac{\rho_\infty'}{\rho_\infty  \sqrt{\left(\rho_\infty'\right)^2+1}} -\Theta\right)\phi'\, dx
\]
where
    $$\Theta(x):= \bigintsss_{a}^x -\frac{\sqrt{\left(\rho_\infty'\right)^2+1}}{\rho_\infty^2}\, d\vartheta.
    $$
    Notice that 
    \[
    |\Theta(x)|\le \frac{1}{r}\min_{X_r}\overline{\F}_\infty
    \]
    and therefore $\Theta$ is bounded and continuous. Applying Du Bois-Reymond Lemma, we deduce that 
\begin{equation}
    \label{eq:integrale_primo_variazionale}
    \frac{\rho_\infty'}{\rho_\infty  \sqrt{\left(\rho_\infty'\right)^2+1}} -\Theta = \Gamma, \quad \text{a.e.\,on $(a,b)$},
\end{equation}
where $\Gamma\in \R$.
If $\varphi \colon [0,+\infty) \to \R$ is given by 
\[
\varphi(x)=:\frac{x}{\sqrt{1+x^2}},
\]
then $\varphi^{-1}\in C^1([0, +\infty))$ and 
\begin{equation}\label{r}
{\rho}_\infty' = f^{-1}\left(\left(\Gamma+ \Theta\right)\rho_\infty\right)  \quad \hbox{a.e. on } (a,b).
\end{equation}
As a consequence, $\rho_\infty'\in W^{1,1}(a,b)$, which means $\rho_\infty\in W^{2,1}(a,b)$. Differentiating \eqref{eq:integrale_primo_variazionale}, we get  
\begin{equation}
      \label{e:rho''_rho_infty}
    \rho_\infty  \rho_\infty''+(\rho_\infty')^2+1=0, \quad \text{a.e.\,on $(a,b)$}.
\end{equation}
 Hence, $\rho_\infty\in C^2(a,b)$ and 
    \begin{equation*}
        \label{e:EL_infty}
        \rho_\infty''=-\frac{1+(\rho_\infty')^2}{\rho_\infty}<0,
    \end{equation*}
    implying that $\rho_\infty$ is strictly concave. Notice that equation \eqref{r} implies also that $\rho_\infty \in C^1(-h,h)$. Moreover, by concavity, it must be $\rho_\infty>r$ on $(-h,h)$.\\
    \\
    {\em Step 3.} We 
    check that boundary conditions are satisfied, namely $\rho_\infty(\pm h) = r$. Notice that by concavity the traces of $\rho_\infty$ are well defined in the usual Sobolev sense. We assume by contradiction that $\rho_\infty(h)\in(0,r)$. Let $\eta \in (0,h)$ and let $\zeta \in (\rho_\infty(h),r)$. Let $\rho_\eta\in X_r$ be given by 
\[
\rho_\eta(x)=\left\{
\begin{array}{ll}
\rho_\infty(x)& \text{if $x\in (-h,h-\eta)$}\\
\\
\displaystyle\frac{\zeta-\rho_\infty(h-\eta)}{\eta}(x-h)+\zeta& \text{if $x\in [h-\eta,h)$}
\end{array}\right.
\]
and let $p\colon(0,h)\to \R$ be  $p(\eta):=\overline{\mathcal{F}}_\infty(\rho_\eta)$. Since $\rho_\eta \to \rho_\infty$ uniformly on any $[a,b]\subset (-h,h)$, we get that $p$ is continuous on $[0,h)$ with $p(0) = \overline{\mathcal{F}}_\infty(\rho_\infty)$ and $p$ is differentiable on $(0,h)$.
By definition, for any $\eta>0$ we have 
\[
p(\eta)=\underbrace{\bigintsss_{-h}^{h-\eta}\frac{\sqrt{1+(\rho_\infty')^2}}{\rho_\infty}\,dx}_{\mathcal{T}_1}+\underbrace{\bigintsss_{h-\eta}^h
\frac{\sqrt{1+(\rho_\eta')^2}}{\rho_\eta}\,dx}_{\mathcal{T}_2}+\abs{\log(\rho_\infty(-h)) - \log r} + \abs{\log \zeta - \log r}.
\]
As a consequence we have 
\[
\lim_{\eta\to 0}\mathcal{T}_1'(\eta)=-\frac{\sqrt{1+(\rho'_\infty(h))^2}}{\rho_\infty(h)}\in [-\infty,0),
\] 
noticing that $\rho_\infty'(h) \in [-\infty,0)$ exists by the concavity of $\rho_\infty$ deduced in Step 3. Concerning $\mathcal{T}_2$, we have
\begin{align*}
\mathcal{T}_2(\eta) &=\bigintsss_{h-\eta}^h\frac{\sqrt{1+(\rho_\eta')^2}}{\rho_\eta}\,dx=
\frac{\sqrt{(\zeta-\rho_\infty(h-\eta))^2+\eta^2}}{\zeta-\rho_\infty(h-\eta)}\left(\log \zeta-\log\rho_\infty(h-\eta)\right)
\end{align*}
from which 
\begin{align*}
\lim_{\eta\to 0}\mathcal{T}_2'(\eta)=\frac{\rho_\infty'(h)}{\rho_\infty(h)}\in [-\infty,0).
\end{align*}
If $\rho_\infty'(h)=-\infty$, then 
\[
\lim_{\eta\to 0}p'(\eta)=-\infty.
\]
Otherwise, 
$$
p'(0) = \frac{1}{\rho_\infty(h)}\left(\rho'_\infty(h)- \sqrt{1+ (\rho'_\infty(h))^2}\right) <0.
$$
In any case, this contradicts the minimality of $\rho_\infty$. Then $\rho_\infty(h)=r$. Using the same argument we can prove that $\rho_\infty(-h)=r$.
\\
\\
{\em Step 4.} Finally, putting together all the information from the previous steps we can conclude: $\rho_\infty$ is a minimizer of $\F_\infty$ on
\[
\left\{\rho \in W^{1,1}_{\rm loc}(-h,h)\cap C^0([-h,h]) : \rho\ge r,\,\rho(\pm h)=r\right\}.
\]
Observe that for any $\rho\in X$ we have $\mathcal{F}_\infty(\rho \vee r) \leq \mathcal{F}_\infty(\rho)$. Indeed, if 
$$A=\left\{x\in (-h,h) : \rho(x)<r\right\},$$ 
then 
\[
\begin{aligned}
\mathcal{F}_\infty(\rho)&=\bigintsss_a^b \frac{1}{\rho}\sqrt{1+(\rho')^2} \,dx\\
&\ge \bigintsss_A \frac{1}{\rho}\,dx+\bigintsss_{(-h,h)\setminus A}\frac{1}{\rho}\sqrt{1+(\rho')^2} \,dx\\
&\ge \bigintsss_A \frac{1}{r}\,dx+\bigintsss_{(-h,h)\setminus A}\frac{1}{\rho}\sqrt{1+(\rho')^2} \,dx=\mathcal F_\infty(\rho\vee r).
\end{aligned}
\]
This implies that $\rho_\infty$ minimizes $\F_\infty$ also on the set 
\[
\left\{\rho \in W^{1,1}_{\rm loc}(-h,h)\cap C^0([-h,h]) : \rho(\pm h)=r\right\}.
\]
As a consequence, the Euler-Lagrange equation \eqref{e:rho''_rho_infty} holds true on $(-h,h)$. Moreover, since the energy functional does not explicitly depend on $x$, $\F_\infty$ has the first integral 
\begin{equation}
    \label{e:integrale_primo_rhoinfty}
    \frac{1}{\rho_\infty \sqrt{1+(\rho_\infty')^2}}=E_\infty,
\end{equation}
where $E_\infty$ is a positive constant. Moreover, \eqref{e:integrale_primo_rhoinfty} can be rewritten as 
\[
\rho_\infty^2(1+(\rho_\infty')^2)=\frac{1}{E_\infty^2}.
\]
Let $v=\rho_\infty^2$. Then 
\[
4v+(v')^2=\frac{4}{E_\infty^2}.
\]
Since $v'\geq 1$, we have $v\le \sfrac{1}{E_\infty^2}$. Setting 
\[
v=\frac{1}{E_\infty^2}-u^2,
\]
we obtain $(u')^2=1$. Since both $v$ and $u$ are smooth functions, either $u(x)=x+k$ or $u(x)=-x+k$. In both cases,  boundary conditions $\rho_\infty(\pm h)=r$ give $k=0$, implying that $u^2=x^2$ and that
\[
E_\infty=\frac{1}{\sqrt{h^2+r^2}}.
\]
Hence, we obtain
$$\rho_\infty(x)=\sqrt{h^2+r^2-x^2}$$ 
yielding the conclusion since $\rho_\infty\in W^{1,1}(-h,h)$, and thus $\rho_\infty\in X$.
\end{proof}
\begin{remark}
We point out that the functional $\F_\infty$ arises in the study of the geodesics on the  Poincar\'e half-plane. It is well known that the geodesics are either vertical half-lines or arcs of circumference meeting orthogonally on the $x$-axes. Nevertheless, in classical differential geometry geodesics are just smooth critical points of the energy functional, not necessarily minimizers of the energy functional in a suitable Sobolev space.
\end{remark}

Next, we prove that for any $c>0$, it is energetically convenient to stay below $\rho_\infty$.

\begin{lemma}\label{tilli2}
For all $c>0$ and for all $\rho \in X$, we have $\mathcal F_c(\rho \wedge \rho_\infty)\le \mathcal F_c(\rho)$.
\end{lemma}
\begin{proof}
    If $\rho\le \rho_\infty$ everywhere, there is nothing to prove. 
    Otherwise, there is $x_0 \in (-h,h)$ such that $\rho(x_0)>\rho_\infty(x_0)$.
    By continuity, such a strict inequality is still valid around $x_0$. Let 
\[
a_\infty=\sup\left\{x \in (-h,x_0): \rho(x)>\rho(x_0)\right\}, \qquad b_\infty=\inf\left\{x \in (x_0,h): \rho(x)>\rho(x_0)\right\}.
\]
By construction, $\rho(a_\infty)=\rho_\infty(a_\infty)$, $\rho(b_\infty)=\rho_\infty(b_\infty)$ and $\rho(x)>\rho_\infty(x)$ for every $x\in (a_\infty,b_\infty)$. To conclude, it is sufficient to prove that 
\begin{equation}\label{stimalocale_E_infty}
\bigintsss_{a_\infty}^{b_\infty}\left(\rho_\infty+\frac{c}{\rho_\infty}\right)\sqrt{1+(\rho_\infty')^2}\,dx \le \bigintsss_{a_\infty}^{b_\infty}\left(\rho+\frac{c}{\rho}\right)\sqrt{1+(\rho')^2}\,dx.
\end{equation}
Since $\rho_\infty$ minimizes $\F_\infty$ we have 
$$
\bigintsss_{a_\infty}^{b_\infty} \frac{\sqrt{1+ (\rho'_\infty)^2}}{\rho_\infty}\, dx \leq \bigintsss_{a_\infty}^{b_\infty} \frac{\sqrt{1+ (\rho')^2}}{\rho}\, dx \qquad \forall\, \rho \in X.
$$
Since $\rho \geq \rho_\infty$ on $(a_\infty, b_\infty)$, in order to prove \eqref{stimalocale_E_infty} it is sufficient to show that 
\begin{equation}
    \label{e:condizione_toshow_tilli2}
    \bigintsss_{a_\infty}^{b_\infty} \rho_\infty\sqrt{1 + (\rho'_\infty)^2}\, dx \leq 
\bigintsss_{a_\infty}^{b_\infty} \rho_\infty \sqrt{1 + (\rho')^2}\, dx.
\end{equation}
Using the convexity of the function $x\mapsto \sqrt{1+x^2}$ and integrating by parts, we deduce that 
\[
\begin{aligned}
    \bigintsss_{a_\infty}^{b_\infty} \rho_\infty\left(\sqrt{1+(\rho')^2}- \sqrt{1+(\rho'_\infty)^2}\right)\, dx&\geq 
\bigintsss_{a_\infty}^{b_\infty}  \frac{\rho_\infty \rho'_\infty}{\sqrt{1 + (\rho_\infty')^2}}(\rho'-\rho_\infty')\, dx \\
&= - \bigintsss_{a_\infty}^{b_\infty}\left(\frac{\rho_\infty \rho'_\infty}{\sqrt{1 + (\rho_\infty')^2}}\right)' (\rho-\rho_\infty)\, dx.
\end{aligned}
\]
Using the expression of $\rho_\infty = \sqrt{h^2+r^2- x^2}$, we get 
$$
\left(\frac{\rho_\infty \rho'_\infty}{\sqrt{1 + (\rho_\infty')^2}}\right)' = \frac{2 x^2 - (h^2+ r^2)}{\sqrt{h^2+r^2- x^2}}\le 0
$$
since $-\sqrt{(h^2+ r^2)}< - h \leq x \leq h < \sqrt{(h^2+ r^2)}$, and this proves \eqref{e:condizione_toshow_tilli2}.
\end{proof}

Then, we prove that it is energetically convenient to stay above $\rho_0$; remember that $\rho_0$ is the unique minimizer of $\F_0$, namely the stable catenary connecting the given boundary data.

\begin{lemma}\label{tilli1}
For all $c>0$ and for all $\rho \in X$ we have $\mathcal F_c(\rho \vee \rho_0)\le \mathcal F_c(\rho)$.
\end{lemma}
\begin{proof}
If $\rho\ge \rho_0$ everywhere there is nothing to prove. Otherwise, we argue as in the previous proof. Let us assume that there is $x_0 \in (-h,h)$ such that $\rho(x_0)<\rho_0(x_0)$. By continuity this strict inequality remains valid around $x_0$. Let 
\[
a_0=\sup\left\{x \in (-h,x_0): \rho(x)<\rho(x_0)\right\}, \qquad b_0=\inf\left\{x \in (x_0,h): \rho(x)<\rho(x_0)\right\}.
\]
By construction, $\rho(a_0)=\rho_0(a_0)$, $\rho(b_0)=\rho_0(b_0)$ and $\rho(x)<\rho_0(x)$ for every $x\in (a_0,b_0)$. Repeating the argument of the previous proof, we need to prove that  
\begin{equation}\label{stimalocale_E_1}
\bigintsss_{a_0}^{b_0}\frac{1}{\rho_0}\sqrt{1+(\rho_0')^2}\,dx \le \bigintsss_{a_0}^{b_0}\frac{1}{\rho_0}\sqrt{1+(\rho')^2}\,dx.
\end{equation}
If $u(x):=\sqrt{1+x^2}$,  condition \eqref{stimalocale_E_1} reads as 
\begin{equation}\label{stimalocale_E_2}
\bigintsss_{a_0}^{b_0}\frac{1}{\rho_0}\left(u(\rho')-u(\rho_0')\right)\,dx\geq 0.
\end{equation}
Using the convexity of $u$ and integrating by parts, we obtain 
\[
\bigintsss_{a_0}^{b_0}\frac{1}{\rho_0}\left(u(\rho')-u(\rho_0')\right)\,dx\geq\bigintsss_{a_0}^{b_0}\frac{1}{\rho_0}u'(\rho_0')(\rho'-\rho_0')\,dx=
-\bigintsss_{a_0}^{b_0}\left(\frac{1}{\rho_0}u'(\rho_0')\right)'(\rho-\rho_0)\,dx.
\]
Since $\rho<\rho_0$ on $(a_0,b_0)$, to get \eqref{stimalocale_E_2}, we must show that 
\begin{equation}\label{stimalocale_E_3}
\left(\frac{1}{\rho_0}u'(\rho_0')\right)'\geq 0.
\end{equation}
Since $\rho_0(x)=\Pi_0\cosh\left(\sfrac{x}{\Pi_0}\right)$, we directly have  
\[
\frac{1}{\rho_0}u'(\rho_0')=\frac{\sinh\left(\sfrac{x}{\Pi_0}\right)}{\Pi_0\left(1+\sinh^2\left(\sfrac{x}{\Pi_0}\right)\right)}.
\]
The function 
\[
(0,+\infty)\ni x\mapsto \frac{\sinh x}{1+\sinh^2 x}
\]
is increasing on $(0,\sinh^{-1}1)$. 

Let $\beta$ be such that
\[
\min_{x>0}\frac{\cosh x}{x}=\frac{\cosh \beta}{\beta}<1.
\]
In particular, since $\sfrac{h}{\Pi_0}<\beta$, we deduce that 
\begin{equation}\label{hpi0}
\sinh \frac{h}{\Pi_0}<\sinh \beta=\frac{\cosh \beta}{\beta}<1.
\end{equation}
Applying \eqref{hpi0}, we obtain
\[
\sinh\frac{x}{\Pi_0}<\sinh\frac{h}{\Pi_0}<1,
\]
which ensures that \eqref{stimalocale_E_3} holds true and this concludes the proof.
\end{proof}

\subsection{Proof of \Cref{prop:esistenza_minimo_Ec}}

In order to prove the existence of minimizers for the energy functional $\mathcal{F}_c$, it is necessary to consider two different cases: $0<c <r^2$ and $c >r^2$. 
\subsubsection{Case 1: \texorpdfstring{$c<r^2$}{uno}}
\label{sec:c<r2}
In this section we take $c\in (0,r^2)$. The argument is similar to the one used in the proof of \Cref{lemma:esiste_minimo_Finfty}. First of all, by \Cref{thm:rilassato}, we define the extended energy functional as
\begin{equation*}
    \overline{\mathcal{F}}_c(\rho)=\mathcal{F}_c(\rho)+\left|F(\rho(-h))-F(r)\right|+\left|F(\rho(h))-F(r)\right|,
\end{equation*}
where $F$ is a primitive of $f = x + \sfrac{c}{x}$.  Precisely, the energy functional $\overline{\F}_c$ reads 
\begin{equation*}
    \overline{\mathcal{F}}_c(\rho)=\mathcal{F}_c(\rho)
+\left|\frac{(\rho(-h))^2}{2}+c\log(\rho(-h))-\frac{r^2}{2}-c\log r\right|
+\left|\frac{(\rho(h))^2}{2}+c\log(\rho(h))-\frac{r^2}{2}-c\log r\right|.
\end{equation*}
\begin{lemma} 
\label{prop:minimo_rilassato}
   Let 
   \[
   X_c= \left\{\rho\in W^{1,1}_{\rm loc}(-h,h):\rho \geq \sqrt{c}\right\}.
   \]
   The variational problem
    $$
    \min_{\rho \in X_c}\overline{\mathcal{F}}_c(\rho)
    $$
    admits a solution. Moreover, if $\rho_c$ is a minimizer of $\overline\F_c$, then there exists $x_0\in (-h,h)$ such that $\rho_c(x_0) > \sqrt{c}$.
\end{lemma}
\begin{proof}
The existence of a minimizer follows just by applying \Cref{thm:minimo_rilassato} with $g=f_{\big|_{[\sqrt c,+\infty)}}$: indeed the function $g$ has a unique critical point $\overline{x} = \sqrt c$, and 
        $$
        g''(\sqrt c) = \frac{2}{\sqrt{c}} >0,
        $$
        satisfying $ii)$ of \Cref{thm:minimo_rilassato}. Let $\rho_c$ be a minimizer of $\overline\F_c$. Assume by contradiction that $\rho_c=\sqrt c$. For any $t\ge 0$ we have $\rho_c+t\in X_c$ and 
   \[
   \overline\F_c(\rho_c+t)=2h\left(\sqrt c+t+\frac{c}{\sqrt c+t}\right)+r^2-(\sqrt c+t)^2+2c\log r-2c\log(\sqrt c+t).
   \]
   Then 
   \[
   \frac{d}{dt}\overline\F_c(\rho_c+t)_{\big|_{t=0}}=-4\sqrt c<0
   \]
   which contradicts the fact that $\rho_c$ is a minimizer.
\end{proof}

\begin{remark}
We remark that, by \Cref{prop:minimo_rilassato}, the function $\rho_c = \sqrt{c}$ is not a critical point of $\overline{\F}_c$ whenever $c>0$. This is in contrast with the case $c=0$, where the problem reduces to that of minimal surfaces of revolution. Indeed, in that setting, the Goldschmidt solution $\rho = 0$ is a local minimizer of $\overline{\F}_0$ (see \cite[Step 6 of Theorem 3.3]{bllm}).
\end{remark}

To pass from a minimizer of the extended energy functional $\overline{\F}_c$ to a minimizer for the energy functional $\mathcal{F}_c$, we need more regularity.

\begin{lemma}
\label{lemma:variazione_prima}
    Let $\rho_c$ be a minimizer for $\overline{\mathcal{F}}_c(\rho)$ on $X_c$. Then $\rho_c \in C^1(-h,h)$. Moreover, let $(a,b)\subset (-h,h)$ be an interval with $\rho_c(x)>\sqrt c$ for any $x\in (a,b)$. Then $\rho_c \in C^2(a,b)$ and 
\begin{equation}
    \label{e:rho''}
    \rho_c''=\frac{((\rho_c)^2-c)(1+(\rho_c')^2)}{\rho_c((\rho_c)^2+c)}.
\end{equation}
    In particular, $\rho_c''>0$ for all $x \in (a,b)$, namely $\rho_c$ is a strictly convex function on $(a,b)$. 
\end{lemma}
\begin{proof}
Since $\rho_c$ is a minimizer for the extended energy functional $\overline{\mathcal{F}}_c(\rho)$, we can compute free outer variations in $(a,b) \subset (-h,h)$.
Let $\psi \in C^\infty_c(a,b)$ and let $\sigma >0$ be small enough. Then, for all $t \in (-\sigma, \sigma)$, we have $\rho_c + t \psi \in X_c$. Thus, we get
$$
0 = \frac{d}{dt}\overline{\mathcal{F}}_c(\rho_c+t\psi)\at[\Big]{t = 0}{}=
\bigintsss_{a}^b \psi \left(1- \frac{c}{(\rho_c)^2}\right)\sqrt{1+(\rho_c')^2}\, dx+
\bigintsss_{a}^b \frac{\psi'\rho_c'}{\sqrt{1+(\rho_c')^2}} \left(\rho_c + \frac{c}{\rho_c}\right)\, dx.
$$
Let
\begin{equation}
    \label{e:Theta}
    \left\{\begin{aligned}
    (a, b)&\to \R\\
    x&\mapsto \Theta(x):=\bigintsss_{a}^x \left(1- \frac{c}{(\rho_c)^2}\right)\sqrt{1+(\rho_c')^2}\, d\vartheta.
\end{aligned}
\right.
\end{equation}
Then, we obtain
$$
\bigintsss_{a}^b\left(\frac{\rho_c'}{\sqrt{1+(\rho_c')^2}} \left(\rho_c + \frac{c}{\rho_c}\right)-\Theta\right)\psi'\, dx=0.
$$
By Du Bois-Reymond lemma, we deduce that
\begin{equation}
\label{e:EL}
    \frac{\rho_c'}{\sqrt{1+(\rho_c')^2}} \left(\rho_c + \frac{c}{\rho_c}\right)-\Theta=\Gamma \quad \hbox{a.e.\,on } (a,b),
\end{equation}
where $\Gamma$ is a constant. Let $\varphi \colon [0,+\infty) \to \R$ be given by 
$$
\varphi(x):= \frac{x}{\sqrt{1+x^2}}.
$$
Then, 
since $\varphi^{-1}\in C^1(0, +\infty)$, we get
\begin{equation}\label{r1}
\rho_c' = \varphi^{-1}\left(\frac{\Gamma+ \Theta}{\rho_c + \sfrac{c}{\rho_c}}\right)  \quad \hbox{a.e.\,on } (a,b).
\end{equation}
Thus, $\rho_c' \in W^{1,1}(a,b)$ implying that $\rho_c \in W^{2,1}(a,b)$ and since we are in a one-dimensional setting, $\rho_c \in C^1([a,b])$. Moreover, we remark that equation \eqref{r1} implies also that $\rho_c \in C^1(-h,h)$. Integrating \eqref{e:EL} and using the expression of $\Theta$ in \eqref{e:Theta}, we get
$$
(1+(\rho_c')^2)\left(1-\frac{c}{(\rho_c)^2}\right)-\left(\rho_c+\frac{c}{\rho_c}\right)\rho_c''=0
$$
getting 
\[
    \rho_c''=\frac{((\rho_c)^2-c)(1+(\rho_c')^2)}{\rho_c((\rho_c)^2+c)}
\]
which is \eqref{e:rho''}, and this yields the conclusion. 
\end{proof}
We are finally in position to prove \Cref{prop:esistenza_minimo_Ec}.

\begin{proof}[Proof of \Cref{prop:esistenza_minimo_Ec}]
Let $\rho_c$ be a minimizer for $\overline{\mathcal{F}}_c$ on $X_c$. 
We divide the proof in two steps.\\
\\
{\sl Step 1.} We claim that there are no points $x_0\in (-h,h)$ such that $\rho_c(x_0) = \sqrt{c}$. Let us assume by contradiction that there exists $x_0\in (-h,h)$ such that $\rho_c(x_0) = \sqrt{c}$. Without loss of generality, we can assume $\rho_c(x)>\sqrt c$ on $(x_0-\delta,x_0)$ for some $\delta>0$.
For sure $\rho_c'(x_0)= 0$ since $\rho_c \in C^1(-h,h)$. However, neither $\rho_c'(x_0)$ can vanish. Indeed, if we consider the Cauchy problem 
\[
\left\{\begin{array}{ll}
\displaystyle \rho_c''=\frac{(\rho_c^2-c)(1+(\rho_c')^2)}{\rho_c(\rho_c^2+c)},\\
\\
\rho_c(x_0)=\sqrt c,\\
\\
\rho_c'(x_0)=0,
\end{array}\right.
\]
this has two different local solutions for $x\ge x_0$: $\rho_c$ and the constant $\sqrt c$,  which is a contradiction. Hence, $\rho_c > \sqrt{c}$ for all $x \in (-h,h)$.\\
\\
{\sl Step 2.} We claim that $\rho_c(\pm h) = r$ in the Sobolev sense. Combining \Cref{lemma:variazione_prima} with Step 1, we can deduce that the function $\rho_c$ is convex on $(-h,h)$. In particular, the traces $\rho_c(\pm h)$ exist in the usual Sobolev sense. Furthermore, $\rho_c(\pm h)\ge \sqrt c$.
We prove now that $\rho_c(h)=r$; the other case follows similarly. 
We assume by contradiction that $\rho_c(h)\in [\sqrt c,r)$. Let $\eta \in (0,h)$ and let $\zeta \in (\rho_c(h),r)$.
Let $\rho_\eta$ be given by 
\[
\rho_\eta(x)=\left\{\begin{array}{ll}
\rho_c(x)& \text{if $x\in (-h,h-\eta)$}\\
\\
\displaystyle\frac{\zeta-\rho_c(h-\eta)}{\eta}(x-h)+\zeta& \text{if $x\in [h-\eta,h)$}
\end{array}\right.
\]
and let $p\colon(0,h)\to \R$ be given by $p(\eta)=\overline{\mathcal{F}}_c(\rho_\eta)$. 
By definition, for any $\eta>0$ we have 
\begin{equation}
    \label{e:Geta}
    \begin{aligned}
         p(\eta)&=\underbrace{\int_{-h}^{h-\eta}\left(\rho_c+\frac{c}{\rho_c}\right)\sqrt{1+(\rho'_{\color{black}c\color{black}})^2}\,dx}_{\mathcal{T}_1}+\underbrace{\int_{h-\eta}^h\left(\rho_\eta+\frac{c}{\rho_\eta}\right)\sqrt{1+(\rho_\eta')^2}\,dx}_{\mathcal{T}_2}\\
&+\left|\frac{\rho_{\color{black}c\color{black}}(-h)^2}{2}+c\log\rho_{\color{black}c\color{black}}(-h)-\frac{r^2}{2}-c\log r\right|+\left|\frac{\zeta^2}{2}+c\log \zeta-\frac{r^2}{2}-c\log r\right|.
    \end{aligned}  
\end{equation}
Since $\rho_\eta \to \rho_c$ uniformly on any $[a,b]\subset (-h,h)$, we deduce that we have
    $$\lim_{\eta\to 0^+}\overline{\mathcal{F}}_c(\rho_\eta)=\overline{\mathcal{F}}_c(\rho_c),$$
   implying that the function $p$ is continuous on $[0,h)$ with $p(0)=\overline{\mathcal{F}}_c(\rho_c)$. Moreover, the function $p$ is differentiable on $(0,h)$. Differentiating each term of \eqref{e:Geta}. We obtain
\[
\mathcal{T}_1'(0)=-\left(\rho_c(h)+\frac{c}{\rho_c(h)}\right)\sqrt{1+(\rho'_c(h))^2},
\] 
where we stress that $\rho'_c(h)$ exists and it is finite by \Cref{lemma:variazione_prima} and by \Cref{tilli1}. 
Computing $\mathcal{T}_2$, we get
\begin{align*}
\mathcal{T}_2(\eta) &=\bigintsss_{h-\eta}^h\left(\rho_\eta+\frac{c}{\rho_\eta}\right)\sqrt{1+(\rho_\eta')^2}\,dx\\ 
&=\frac{\sqrt{(\zeta-\rho_c(h-\eta))^2+\eta^2}}{\zeta-\rho_c(h-\eta)}\left(\frac{\zeta^2}{2}-\frac{\rho_c(h-\eta)^2}{2}+c\log \zeta-c\log\rho_c(h-\eta)\right)
\end{align*}
from which 
\begin{align*}
\mathcal{T}_2'(0)= \frac{\left(c + (\rho_c(h))^2\right)\rho_c'(h)}{\rho_c(h)}.
\end{align*}
Then, collecting everything, it holds
$$
p'(0) = \left(\rho_c(h)+\frac{c}{\rho_c(h)}\right)\left(-\sqrt{1+(\rho'_c(h))^2} + \rho'_c(h)\right)<0
$$
contradicting the minimality of $\rho_c$.
\\
\\
{\sl Step 3}. We can now conclude the proof. We can say that $\rho_c$ is a minimizer of $\F_c$ on
\[
\left\{\rho \in W^{1,1}_{\rm loc}(-h,h)\cap C^0([-h,h]) : \rho\ge \sqrt c,\,\rho(\pm h)=r\right\}.
\]
Observe that for any $\rho\in X$, we have $\mathcal{F}_c(\rho \vee \sqrt c) \leq \mathcal{F}_c(\rho)$. Indeed, if 
$$A=\left\{x\in (-h,h) : \rho(x)<\sqrt c\right\},$$
then 
\[
\begin{aligned}
\mathcal{F}_c(\rho)&=\bigintsss_a^b f(\rho)\sqrt{1+(\rho')^2} \,dx\\
&\ge \bigintsss_A f(\rho)\,dx+\bigintsss_{(-h,h)\setminus A}f(\rho)\sqrt{1+(\rho')^2} \,dx\\
&\ge \bigintsss_A 2\sqrt c\,dx+\bigintsss_{(-h,h)\setminus A}f(\rho)\sqrt{1+(\rho')^2} \,dx=\mathcal F_c(\rho\vee \sqrt c).
\end{aligned}
\]
Then $\rho_c$ minimizes $\F_c$ also on the set 
\[
\left\{\rho \in W^{1,1}_{\rm loc}(-h,h)\cap C^0([-h,h]) : \rho(\pm h)=r\right\}.
\]
By \Cref{tilli1}, $\rho_c \ge \rho_0$, implying that actually $\rho_c \in X$, and this ends the proof.
\end{proof}

\subsubsection{Case 2: \texorpdfstring{$c>r^2$}{due}}
\label{sec:c>r^2}
The analysis of the previous section can be adapted to the case $c>r^2$. Indeed, it is sufficient to minimize the functional $\overline{\mathcal{F}}_c$ on the set 
   \[
   X_c= \left\{\rho\in W^{1,1}_{\rm loc}(-h,h):\rho \leq \sqrt{c}\right\}.
   \]
   In the relaxation process of $\F_c$, here we are considering as $f$ the function $f_{\big|_{[c_0,+\infty)}}$, where we recall that $c_0=\min\rho_0$. This choice prevents the fact that $f$ blows up when $x\to 0^+$. In this way, the variational problem
    $$
    \min_{\rho \in X_c}\overline{\mathcal{F}}_c(\rho)
    $$
    admits a solution $\rho_c \in X_c$. From now on, one can repeat the proof along the same lines. We deduce that $\rho_c<\sqrt c$ and this permits to say that $\rho_c$ is strictly concave computing a first variation around a point where $\rho_c<\sqrt c$. 
    We then perturb the boundary values proving that actually $\rho_c(\pm h)=r$ in the Sobolev sense: here, it is crucial to use \Cref{tilli2} that guarantees that $\rho_c'(\pm h)$ is finite. Finally, it holds 
    \[
    \F_c(\rho\vee \sqrt c)\le \F_c(\rho)
    \]
    for any $\rho \in X$ whenever $c>r^2$ and then $\rho_c\in X$ is a minimizer of $\F_c$ and $\rho_c \in C^2([-h,h])$.

\begin{remark}
\label{rem:integrale_primo}
    As a consequence of \Cref{prop:esistenza_minimo_Ec}, in both cases $c<r^2$ and $c>r^2$, we have obtained that there exists a minimizer $\rho_c \in W^{1,1}(-h,h) \cap C^2([-h,h])$ satisfying the boundary conditions $\rho_c(\pm h) = r$. Moreover, performing an outer first variation, we can say that $\rho_c $ is a strictly positive solution of the boundary value problem 
    \begin{equation}
    \label{bvp1}
        \left\{
        \begin{aligned}
            &\rho''=\frac{(\rho^2-c)(1+(\rho')^2)}{\rho(\rho^2+c)},\\
            &\rho(-h)=r,\\
            &\rho(h)=r.
        \end{aligned}\right.
    \end{equation}
We will investigate such a problem in the next section and this will complete the proof of \Cref{thm_existence}.
\end{remark}

\section{Solutions of a boundary value problem}
\label{sec:studio_bvp}

By \Cref{rem:integrale_primo}, to better geometrically characterize the minimizer of $\mathcal{F}_c$ on $X$, we need to study the 
second order boundary value problem coming from the Euler-Lagrange equations, namely \eqref{bvp1}.

\begin{proposition}
\label{prop:proprieta_bvp}
Let $\rho \in C^2([-h,h])$ be a strictly positive solution of the boundary value problem 
\begin{equation}\label{bvp}
\left\{\begin{array}{ll}
\displaystyle \rho''=\frac{(\rho^2-c)(1+(\rho')^2)}{\rho(\rho^2+c)}\\
\\
\rho(-h)=\rho(h)=r.
\end{array}\right.
\end{equation}
Then the following facts hold true.
\begin{itemize}
\item[a)] If $c<r^2$, then 
\begin{equation}\label{rho-}
\rho(x)=\sqrt{ \frac{E}{2}\sqrt{E^2-4c}\cosh\frac{2x}{E}+\frac{E^2}{2}-c}
\end{equation}
where $E= E(c)$ is a positive constant that satisfies 
\begin{equation}\label{p-}
\cosh\frac{2h}{E}=\frac{2r^2-E^2+2c}{E\sqrt{E^2-4c}}.
\end{equation}
Moreover, $\rho$ is even, strictly convex and $\rho_0<\rho<r$ on $(-h,h)$.
\item[b)] If $c>r^2$, then 
\begin{equation}\label{rho+}
\rho(x)=\sqrt{-\frac{E}{2}\sqrt{E^2-4c}\cosh\frac{2x}{E}+\frac{E^2}{2}-c}
\end{equation}
where $E= E(c)$ is a positive constant that satisfies 
\begin{equation}\label{p+}
\cosh\frac{2h}{E}=-\frac{2r^2-E^2+2c}{E\sqrt{E^2-4c}}.
\end{equation} 
Moreover, $\rho$ is even, strictly concave and $r<\rho$ on $(-h,h)$.
\end{itemize}
Moreover, both in cases a) and c), it holds 
\begin{equation}\label{Ec}
E^2-4c>0.
\end{equation}
\end{proposition}

\begin{proof}
    We divide the proof into some steps.\\
    \\
    {\it Step 1.} First of all, the following holds: 
    \begin{itemize}
\item[i)] If $c<r^2$, then $\rho>\sqrt c$ everywhere. In particular, $\rho$ is strictly convex.
\item[ii)] If $c>r^2$, then $\rho<\sqrt c$ everywhere. In particular, $\rho$ is strictly concave.
\end{itemize}
Let us show i) and we remark that using the same technique ii) holds true. Let $c < r^2$. We proceed by contradiction. Let $x_0 \in (-h, h)$ be the first point such that $\rho(x_0) = \sqrt{c}$. If $\rho'(x_0) \neq 0$, then necessarily $\rho'(x_0) < 0$, implying that there exists $x_1 > x_0$ such that $\rho(x_1) < \sqrt{c}$.
Consequently, $\rho$ is concave in a neighborhood of $x_1$. This yields a contradiction because, to satisfy the boundary condition at $h$, $\rho$ would need to become convex while remaining strictly below $\sqrt{c}$. 
Thus, $\rho'(x_0) = 0$ and $\rho\ge \sqrt c$ everywhere. But also in this case we reach a contradiction. 
Indeed, in this case the Cauchy problem
\[
\left\{\begin{array}{ll}
\displaystyle \rho''=\frac{(\rho^2-c)(1+(\rho')^2)}{\rho(\rho^2+c)}\\
\\
\rho(x_0)=\sqrt c\\
\\
\rho'(x_0)=0
\end{array}\right.
\]
has two different local solutions for $x\ge x_0$: $\rho$ and the constant $\sqrt c$.
\\
\\
{\it Step 2.} By \Cref{rem:integrale_primo}, \eqref{bvp} comes from the Euler-Lagrange equation of $\mathcal{F}_c$. Moreover, we notice that $\mathcal{F}_c$ does not explicitly depend on $x$, then by a direct computation, it holds
\[
\left(\frac{\displaystyle \rho+\sfrac{c}{\rho}}{\sqrt{1+(\rho')^2}}\right)'=0.
\]
As a consequence, there exists a constant $E>0$, which may depend on $c$, such that 
\begin{equation}
\label{e:integrale_primo}
\rho+\frac{c}{\rho}=E\sqrt{1+(\rho')^2}.
\end{equation}
From now on, we will refer only to the case $c<r^2$. The case $c>r^2$ can be treated in a similar way.
Since $\rho(-h)=\rho(h)$, then there exists $x_0\in (-h,h)$ such that $\rho'(x_0)=0$. Hence, since by Step 1 $\rho>\sqrt c$ when $c<r^2$, we have  
\begin{equation}
    \label{e:stima_E}
    E=\rho(x_0)+\frac{c}{\rho(x_0)}>2\sqrt c.
\end{equation}
In particular, $E^2-4c>0$.  Let
\begin{equation}
    \label{e:exp_v}
    v=\frac{\rho^2+c-\sfrac{E^2}{2}}{\sqrt{E^2-4c}}.
\end{equation}
Substituting \eqref{e:exp_v} into \eqref{e:integrale_primo}, a straightforward computation shows that 
\[
v^2=\frac{E^2}{4}\left(1+(v')^2\right).
\]
As a consequence, either $v\le -\sfrac{E}{2}$ or $v\ge \sfrac{E}{2}$. By Step 1, since $\rho>\sqrt c$ if $c<r^2$, we get that 
\[
v>\frac{2c-\sfrac{E^2}{2}}{\sqrt{E^2-4c}}=-\frac{\sqrt{E^2-4c}}{2}>-\frac{E}{2},
\]
thus $v\ge \sfrac{E}{2}$. Hence, there exists a function $u\in C^2(-h,h)\cap C^0([-h,h])$ such that 
\[
v(x)=\frac{E}{2}\cosh u(x).
\]
In particular, $u$ solves the equation
\[
\cosh^2 \left(u(x)\right)=1+\frac{E^2}{4}\sinh^2 u(x)\left(u'(x)\right)^2,
\]
which reduces to 
\[
\sinh^2 u(x)\left(\frac{E^2}{4}(u'(x))^2-1\right)=0.
\]
Since $u$ is smooth, it must be 
\[
u(x)=\frac{2}{E}x+k
\]
for some $k\in \R$. Then, using \eqref{e:exp_v}, we have
\[
\rho^2(x)=\frac{E}{2}\sqrt{E^2-4c}\cosh\left(\frac{2x}{E}+k\right)+\frac{E^2}{2}-c.
\]
Since $\rho(\pm h)=r$, it follows that $k=0$. Hence, since $\rho$ must be positive, we end up with
\[
\rho(x)=\sqrt{\frac{E}{2}\sqrt{E^2-4c}\cosh\frac{2x}{E}+\frac{E^2}{2}-c}.
\] 
Finally, we notice that \eqref{p-} is a simple consequence of applying the boundary conditions $\rho(\pm h)=r$. Moreover, the other properties are a direct consequence of the explicit expression of $\rho$.
\end{proof}

From now on, let us denote by $\rho_c$ a function of the form \eqref{rho-} or  $\eqref{rho+}$ respectively if $c<r^2$ and $c>r^2$. We now investigate some properties of $\rho_c$. 

\begin{lemma}
\label{lemma:monotonia_inc}
The quantity $\rho_c(0)$ is strictly increasing in $c$. In particular, if $0\le c_1<c_2$ then $\rho_{c_1}<\rho_{c_2}$ on $(-h,h)$.
\end{lemma}

\begin{proof}
Let $c<r^2$. Since the function 
\[
[0,+\infty) \ni c \mapsto \frac{2r^2-E^2+2c}{E\sqrt{E^2-4c}}
\]
is strictly increasing, then assigned a value $c >0$, \eqref{p-}, solved with respect to $E$, has solutions
\[
E(c)^1,E(c)^2,\dots,E(c)^{n_c}, \quad n_c\in \mathbb N.
\]
Precisely, $E(c)^j$ are strictly increasing functions in $c$. Moreover, each $E(c)^j$ gives a corresponding $\rho_c(0)^j$ that satisfies
\[
\rho_c(0)^j+\frac{c}{\rho_c(0)^j}=E(c)^j.
\]
  Hence, since the function 
  $$f: \left\{
  \begin{aligned}
      (\sqrt c,+\infty) &\to (2\sqrt c,+\infty)\\
      x&\mapsto x+\frac{c}{x}
  \end{aligned}
  \right.$$ 
  is a strictly increasing bijection and $E(c)^j>2\sqrt c$ for all $j$  by \eqref{e:stima_E}, we get the thesis: $\rho_c(0)^j$ is a strictly increasing function in $c$. 
\end{proof}

We are now in position to prove the main result of this section: both the transcendental equations 
\begin{align}
\label{e:caso_c<r^2}
    \hbox{if } c<r^2 \qquad &\cosh\frac{2h}{E}=\frac{2r^2-E^2+2c}{E\sqrt{E^2-4c}}\\
    \label{e:caso_c>r^2}
   \hbox{if } c>r^2 \qquad &\cosh\frac{2h}{E}=\frac{-2r^2+E^2-4c}{E\sqrt{E^2-4c}}
\end{align}
admit a unique solution $\smileacc{E}(c)>\Pi_0(h,r)$ and $\frownacc{E}(c)>0$ respectively.
Let us start with the convex case.

\begin{lemma}
\label{lemma:unicita_soluzioni_1}
Let $h,r>0$ be such that $\sfrac{h}{r}\in (0,\omega]$ and let $c\in (0,r^2)$. Then, there exists a unique
$$\smileacc{E}_c(h,r)>\Pi_0(h,r)$$
which solves \eqref{e:caso_c<r^2}.
\end{lemma}

\begin{proof}
For simplicity of notation, let us define 
\[
e = \frac{E}{r}, \qquad \gamma = \frac{c}{r^2}.
\]
For $\gamma \in [0,1)$, let 
\begin{equation*}
    \Psi_\gamma: \left\{
    \begin{aligned}
        (2\sqrt\gamma,\gamma+1] &\to \R\\
        e&\mapsto \Psi_\gamma(e):=\frac{e}{2}{\rm arccosh}\frac{2\gamma - e^2 +2}{ e\sqrt{e^2 - 4\gamma}}.
    \end{aligned}
    \right.
\end{equation*}
First of all, we notice that $\Psi_{\gamma_1}<\Psi_{\gamma_2}$ on $(2\sqrt{\gamma_2},\gamma_1+1)$ whenever $\gamma_1<\gamma_2$.
Indeed, it is sufficient to notice that
$$
\frac{\partial}{\partial \gamma} \Psi_\gamma = \frac{(1-\gamma ) e}{\left(e^2-4 \gamma \right)
   \sqrt{(\gamma +1)^2-e^2}} >0.
$$ 
Moreover, considering $\gamma = 0$, the equation $\Psi_0(e)=\sfrac{h}{r}$ has two solutions: we denote by $e_0(h,r)$ the largest one.
To make the notation easier, from now on we will omit the dependence on $h,r$. To get the thesis, we have to show that the equation 
\begin{equation}
    \label{e:tesi}
    \frac{h}{r}=\Psi_\gamma(e) \qquad \hbox{has a unique solution} \quad \smileacc{e}_\gamma>e_0.
\end{equation}
Indeed, if \eqref{e:tesi} holds true, then we have
\[
\smileacc{E}_c:=r\smileacc{e}_{\sfrac{c}{r^2}}>r e_0=\Pi_0.
\]
We divide the rest of the proof into two steps.
\\
\\
    {\it Step 1. Study of $\Psi_\gamma''$.} 
    We claim that for any $\gamma \in (0,1)$, the function $\Psi_\gamma$ has a unique inflection point $e_\gamma^F$.
    Moreover, once $e_\gamma^F$ exists, we also show that there exists $\gamma^\ast \in (0,1)$ such that:
    \begin{itemize}
    \item[(a)] for any $\gamma \in (0,\gamma^\ast)$ it holds $e_\gamma^F<e_0$;
    \item[(b)] for any $\gamma \in [\gamma^\ast,1)$ the function $\Psi_\gamma$ is strictly decreasing.
    \end{itemize}
    First of all, let us start computing the first derivative of $\Psi_\gamma$ with respect to $e$, getting
    $$
\begin{aligned}
   \Psi_\gamma'(e)  &= 
    \frac{2 \gamma  (\gamma +1)-e^2}{\left(e^2-4 \gamma \right) \sqrt{(\gamma
   +1)^2-e^2}}+\frac{1}{2} {\rm arccosh}
   \frac{2 \gamma
   -e^2+2}{e\sqrt{e^2-4 \gamma }}\\
   &=\frac{2 \gamma  (\gamma +1)-e^2}{\left(e^2-4 \gamma \right) \sqrt{(\gamma
   +1)^2-e^2}}+\frac{1}{2} \log
   \left(\frac{2 \gamma
   -e^2+2}{e\sqrt{e^2-4 \gamma }} + \sqrt{\left(\frac{2 \gamma
   -e^2+2}{e\sqrt{e^2-4 \gamma }}\right)^2 -1 }\right).
\end{aligned}
    $$
   Then, the second derivative reads  
    $$
    \Psi_\gamma''(e) = -\frac{8 \gamma ^2 (\gamma +1)^3+(1-3 (\gamma -2) \gamma ) e^4+2 (\gamma -5)
   \gamma  (\gamma +1)^2 e^2}{e \left(e^2-4 \gamma \right)^2
   \left((\gamma +1)^2-e^2\right)^{\sfrac{3}{2}}}.
    $$
    It turns out that the equation $\Psi''_\gamma(e)=0$ has a unique solution $e_\gamma^F\in (2\sqrt\gamma,\gamma+1)$ explicitly given by 
   \[
    e_\gamma^F=\sqrt{\frac{\gamma ^4-3 \gamma ^3-9 \gamma ^2-\sqrt{\gamma ^8+18 \gamma ^7+15 \gamma ^6-36
   \gamma ^5-33 \gamma ^4+18 \gamma ^3+17 \gamma ^2}-5 \gamma }{3 \gamma ^2-6 \gamma -1}}.
    \]
    It is not difficult to see that the function $(0,1)\ni \gamma \mapsto e_\gamma^F$ is an increasing function. 
    
    We now claim that there exists a unique $\gamma^\ast \in (0,1)$ such that $\Psi_{\gamma^\ast}'(e_{\gamma^\ast}^F)=0$. Indeed, let us rewrite $\Psi'_\gamma$ as follows
    $$
    \Psi'_\gamma(e) = g_\gamma(e) + h_\gamma(e),
    $$
    where
    \begin{align*}
         &g_\gamma(e)=\frac{2 \gamma  (\gamma +1)-e^2}{\left(e^2-4 \gamma \right) \sqrt{(\gamma
   +1)^2-e^2}}, &&&h_\gamma(e)=-\frac{1}{2} \log
   \left(\frac{2 \gamma
   -e^2+2}{e\sqrt{e^2-4 \gamma }} + \sqrt{\left(\frac{2 \gamma
   -e^2+2}{e\sqrt{e^2-4 \gamma }}\right)^2 -1 }\right).
    \end{align*}
    A standard computation, performed by symbolic Mathematica software ver. 14.3, shows that the functions $(0,1)\ni \gamma \mapsto g_\gamma(e_\gamma^F)$ and $(0,1)\ni\gamma \mapsto h_\gamma(e_\gamma^F)$ are strictly monotone with 
    \[
    \lim_{\gamma\to 0^+}g_\gamma(e_\gamma^F)<0, \qquad
    \lim_{\gamma\to 0^+}h_\gamma(e_\gamma^F)>0, \qquad
    \lim_{\gamma\to 1^-}g_\gamma(e_\gamma^F)>0, \qquad
    \lim_{\gamma\to 1^-}h_\gamma(e_\gamma^F)<0.
    \]
    This implies that there is a unique $\gamma^\ast \in (0,1)$ with $g_\gamma(e_\gamma^F)=h_\gamma(e_\gamma^F)$ leading to the claim.
    Moreover, by a direct numerical computation, one can prove that $e_{0.0026}^F<e_0$ and $\Psi_{0.0026}(e_{0.0026}^F)<0$. This means that $\gamma^\ast<0.0026$ and  $e_\gamma^F<e_0$ for any $\gamma \in (0,\gamma^\ast)$. Finally, assume that there is $\widetilde \gamma \in (\gamma^\ast,1)$ with $\Psi_{\widetilde \gamma}(e_{\widetilde \gamma}^F)>0$. This immediately contradicts the uniqueness of $\gamma^\ast$. As a consequence, $\Psi_\gamma(e_\gamma^F)<0$ for any $\gamma \in (\gamma^\ast,1)$, implying that $\Psi_\gamma$ is strictly decreasing for any $\gamma \in (\gamma^\ast,1)$.
    \\
    \\
    {\it Step 2. Uniqueness of $\smileacc{e}_\gamma$.}
    We can conclude the proof. Assume by contradiction that for some $\gamma \in (0,1)$ the equation $\sfrac{h}{r}=\Psi_\gamma(e)$ has two solutions $\smileacc{e}_\gamma^{(1)}$ and $\smileacc{e}_\gamma^{(2)}$ with $\smileacc{e}_\gamma^{(1)}\ne \smileacc{e}_\gamma^{(2)}$ and with $\smileacc{e}_\gamma^{(i)}>e_0$ for $i=1,2$. In particular, $e_\gamma^F>e_0$. Then $\gamma\ge \gamma^\ast$ but this leads to a contradiction since $\Psi_\gamma$ is an invertible function. Hence, necessarily $\smileacc{e}_\gamma^{(1)}= \smileacc{e}_\gamma^{(2)}$.
\end{proof}

Then, for the concave case $c>r^2$, we have a stronger result.

\begin{lemma}
\label{lemma:unicita_soluzioni_2}
Let $h,r>0$ be such that $\sfrac{h}{r}\in (0,\omega]$ and let $c>r^2$. Then, there exists a unique 
$$\frownacc{E}_c(h,r)>0$$
solution of \eqref{e:caso_c>r^2}.
\end{lemma}

\begin{proof}
As in the previous proof, let 
\[
e = \frac{E}{r}, \qquad \gamma = \frac{c}{r^2}.
\]
For $\gamma >1$, let 
\begin{equation*}
    \Psi_\gamma:\left\{
    \begin{aligned}
        (2\sqrt\gamma,\gamma+1] &\to \R\\
        e&\mapsto \Psi_\gamma(e) :=  \frac{e}{2}{\rm arccosh}\frac{-2\gamma + e^2 -2}{ e\sqrt{e^2 - 4\gamma}}.
    \end{aligned}\right. 
\end{equation*}
Again, we have to show that the equation $\sfrac{h}{r}=\Psi_\gamma(e)$ has a unique solution.
Differently from the previous case, here, we are able to prove directly that $\Psi_\gamma$ is a strictly decreasing function.
We remark that all explicit computations are performed by symbolic Mathematica software ver. 14.3. 
Mimicking Step 1 of \Cref{lemma:unicita_soluzioni_1}, we show that $\Psi_\gamma$ has a unique inflection point $e_\gamma^F$ given by 
\[
e_\gamma^F = \sqrt{\frac{\gamma ^4-3 \gamma ^3-9 \gamma ^2+\sqrt{\gamma ^8+18 \gamma ^7+15 \gamma ^6-36
   \gamma ^5-33 \gamma ^4+18 \gamma ^3+17 \gamma ^2}-5 \gamma }{3 \gamma ^2-6 \gamma -1}}.
\]
Moreover, we can rewrite
$$
\Psi_\gamma'(e_\gamma^F)=f_1(\gamma)-f_2(\gamma)
$$
where 
\[
\begin{aligned}
f_1(\gamma)&:=-\frac{1}{2}{\rm arccosh}(g_1(\gamma)),\\
g_1(\gamma)&:=\frac{\gamma ^4-9 \gamma ^3-3 \gamma ^2+\sqrt{(\gamma -1)^2 \gamma ^2 (\gamma +1)^3 (\gamma +17)}+9 \gamma +2}{\left(3 \gamma ^2-6
   \gamma -1\right) \sqrt{\frac{\gamma ^4-15 \gamma ^3+15 \gamma ^2+\sqrt{(\gamma -1)^2 \gamma ^2 (\gamma +1)^3 (\gamma +17)}-\gamma
   }{3 \gamma ^2-6 \gamma -1}} \sqrt{\frac{\gamma ^4-3 \gamma ^3-9 \gamma ^2+\sqrt{(\gamma -1)^2 \gamma ^2 (\gamma +1)^3 (\gamma
   +17)}-5 \gamma }{3 \gamma ^2-6 \gamma -1}}},\\
f_2(\gamma)&:= \frac{-5 \gamma ^4+3 \gamma ^3+5 \gamma ^2+\sqrt{(\gamma -1)^2 \gamma ^2 (\gamma +1)^3 (\gamma +17)}-3 \gamma }{\sqrt{\frac{-2
   \gamma ^4-3 \gamma ^3+\gamma ^2+\sqrt{(\gamma -1)^2 \gamma ^2 (\gamma +1)^3 (\gamma +17)}+3 \gamma +1}{-3 \gamma ^2+6 \gamma +1}}
   \left(\gamma ^4-15 \gamma ^3+15 \gamma ^2+\sqrt{(\gamma -1)^2 \gamma ^2 (\gamma +1)^3 (\gamma +17)}-\gamma \right)}.
\end{aligned}
\]
It turns out that 
\begin{equation}
\label{e:limite_a1}
    \lim_{\gamma\to 1^+}f_1(\gamma)=\frac{1}{4}(\log 2-2\log(1+\sqrt 3))<0, \qquad \lim_{\gamma\to 1^+}f_2(\gamma)=-\infty,
\end{equation}
and
\begin{equation*}
    \lim_{\gamma\to +\infty}f_1(\gamma)=\lim_{\gamma\to +\infty}f_2(\gamma)=0.
\end{equation*}
Thus, by \eqref{e:limite_a1}, we have that $f_1-f_2>0$ on $(1,1+\delta)$ for some $\delta>0$. 
Moreover, by a change of variable $p\mapsto \sfrac{1}{\gamma}$, we prove that there are no $\gamma \in (1,+\infty)$ such that $f_1'(\gamma)-f_2'(\gamma)=0$.
This implies that $f_1-f_2$ is a strictly decreasing function and tends to $0$ as $\gamma\to+\infty$, hence $f_1-f_2>0$ everywhere. As a consequence, $\Psi_\gamma'(e_\gamma^F)\ne 0$ for any $\gamma \in (1,+\infty)$ implying that $\Psi_\gamma$ must be strictly decreasing on $(1,+\infty)$.
\end{proof}

The following corollary is immediate.

\begin{corollary}\label{cor:bvp}
The boundary value problem \eqref{bvp} has a unique solution $\rho_c \in  C^2([-h,h])$ lying above the catenary $\rho_0$.
\end{corollary}

\section{Shape of \texorpdfstring{$\rho_c$}{rho} as \texorpdfstring{$c \to \infty$}{c}}
\label{sec:c_tinfty}

It remains to show that $\rho_c \to \rho_\infty$ uniformly on $[-h,h]$ as $c \to + \infty$ where $\rho_\infty$ is the unique minimizer of $\mathcal{F}_\infty$  defined in \Cref{lemma:esiste_minimo_Finfty}.

\begin{lemma}{\cite[Lemma 3.2]{greco2012}}\label{lemma_conv}
Let $a,b\in \R$ with $a<b$ and let $(\rho_j)$ be a sequence of convex functions on $(a,b)$ such that $\rho_j\to \rho$ uniformly on every $[\gamma,\delta]\subset (a,b)$. Then $\rho_j' \to \rho'$ in $L^1(\gamma,\delta)$ for any $[\gamma,\delta]\subset (a,b)$.
\end{lemma}

\begin{proposition}
\label{prop:gamma_convergence}
    Let $(c_j)$ be a positive sequence with $c_j \to +\infty$. Let $\rho_j$ be the minimizer of $\mathcal{F}_{c_j}$. Then, $\rho_j \to \rho_\infty$ uniformly on $[-h,h]$ as $j \to +\infty$ .
\end{proposition}

\begin{proof}
We divide the proof in two steps.
\\
\\
{\sl Step 1.} Let $\bar\rho_j$ be a sequence in $X$ converging to $\bar\rho\in X$ strongly in $W^{1,1}(-h,h)$. In particular, $\bar\rho_j\to \bar\rho$ uniformly on $[-h,h]$ and $\bar\rho_j'\to \bar\rho'$ in $L^1(-h,h)$. We claim that 
\begin{equation}\label{gamma}
\lim_{j\to+\infty}\F_\infty(\bar\rho_j)=\F_\infty(\bar\rho).
\end{equation}
Since $\bar\rho_j\to \bar\rho$ in $L^\infty(-h,h)$, it is sufficient to show that 
\[
\sqrt{1+(\bar\rho_j')^2} \to \sqrt{1+(\bar\rho')^2} \quad \text{in $L^1(-h,h)$}.
\]
Using the subadditivity of the square root, we have
\[
\left|\sqrt{1+(\bar\rho_j')^2}- \sqrt{1+(\bar\rho')^2}\right|\le \sqrt{\left|(\bar\rho_j')^2-(\bar\rho')^2\right|}. 
\]
As a consequence, by H\"older inequality, we have
\[
\begin{aligned}
\bigintsss_{-h}^h\left|\sqrt{1+(\rho_j')^2}- \sqrt{1+(\rho')^2}\right|\,dx
&\le\bigintsss_{-h}^h \sqrt{\left|(\rho_j')^2-(\rho')^2\right|}\,dx\\
&=\bigintsss_{-h}^h \sqrt{\left|\rho_j'+\rho'\right|}\sqrt{\left|\rho_j'-\rho'\right|}\,dx\\
&\le \left(\bigintsss_{-h}^h\left|\rho_j'+\rho'\right|\,dx\right)^{\sfrac{1}{2}}\left(\bigintsss_{-h}^h\left|\rho_j'-\rho'\right|\,dx\right)^{\sfrac{1}{2}}\\
&\le C \left(\bigintsss_{-h}^h\left|\rho_j'-\rho'\right|\,dx\right)^{\sfrac{1}{2}}\to 0
\end{aligned} 
\]
which proves \eqref{gamma}. 
\\
\\
{\sl Step 2.} Notice that $\rho_j$ is a minimizer of $\F_{c_j}$ on the set 
\[
Y:= \left\{\rho \in W^{1,1}(-h,h),  \,
\rho_0 \leq \rho\leq \rho_\infty,\, \rho(\pm h)=r\right\}.
\]
By Ascoli Arzela's Theorem, $Y$ is compact with respect to the uniform convergence on $[-h,h]$. Hence $\rho_j$ converges, up to a subsequence (not relabeled), to $\rho \in Y$ uniformly on $[-h,h]$. Since $c_j\to+\infty$, without loss of generality we can assume $\rho_j$ concave. Applying \Cref{lemma_conv} (for concave functions), we get also $\rho_j'\to \rho'$ in $L^1(-h,h)$. For any $\bar\rho\in X$ we have 
\[
\F_\infty(\rho)\stackrel{\eqref{gamma}}{=}\lim_{j\to+\infty}\F_\infty(\rho_j)\le \liminf_{j\to+\infty}\frac{1}{c_j}\mathcal{F}_{c_j}(\rho_j)\le \liminf_{j\to+\infty}\frac{1}{c_j}\mathcal{F}_{c_j} (\bar\rho)\stackrel{\eqref{gamma}}{=}\F_\infty(\bar\rho).
\]
This means that $\rho$ is a minimizer of $\F_\infty$, and thus by \Cref{lemma:esiste_minimo_Finfty}, $\rho=\rho_\infty$. 
\end{proof}

\section{Numerical tests}
\label{sec:numerica}
Finally, in this section, we run some numerical simulations to visualize our theoretical results and to give some additional geometrical insights 
to the unique minimizers $\rho_c$.
In the regime where \eqref{cond_catenaria} holds true, namely for closed coaxial rings, for all $c >0$, we solve the Euler-Lagrange equation 
\begin{equation}\label{bvp_num}
\left\{
\begin{aligned}
    &\rho''=\frac{(\rho^2-c)(1+(\rho')^2)}{\rho(\rho^2+c)}\\
&\rho(-h)=r\\
&\rho(h)=r
\end{aligned}
\right.
\end{equation}
varying the geometrical ratio $\sfrac{h}{r}$.

The numerical simulations confirm the theoretical results obtained in \Cref{thm_existence}.
In \Cref{fig:numerica}, we plot the minimizers $\rho_c$ for different values of the parameter $c>0$: the solutions are even, regular and convex for $c<r^2$, concave otherwise and flat if $c=r^2$. Moreover, for $c>0$, all minimizers $\rho_c$ lie above the catenary and they tend to be an arc of circumference as $c$ goes to infinity, see in \Cref{fig:numerica} the blue line and the black one, respectively.
Moreover, we also notice that solutions are increasing as the parameter $c$ increases: for all ratios $\sfrac{h}{r}$, $\rho_c(0)$ is a monotone function with respect to $c$.

\begin{figure}[htbp]
	\begin{subfigure}{\linewidth}
		\centering
	\includegraphics[width=0.5\textwidth]{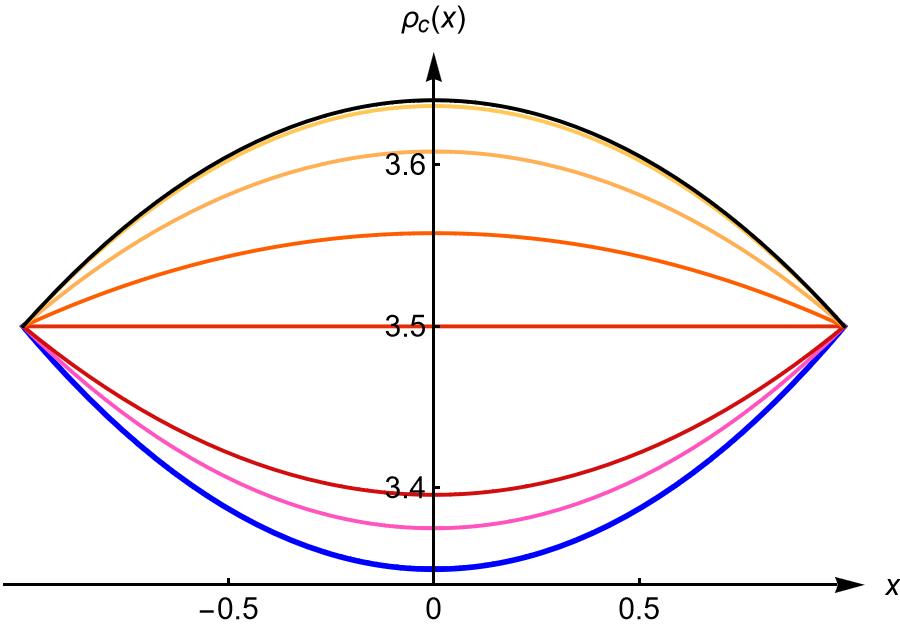}
		\caption{$\rho_c(1) = \rho_c(-1) = \sfrac{7}{2}$}
	\end{subfigure}
	\caption{Numerical representation of minimizers of $\mathcal{F}_c$: we choose $h=1$, as boundary conditions $\rho_c(1) = \rho_C(-1) = \sfrac{7}{2}$ and from the bottom to the top $c = 0, 1,2,12.25(=r^2),30,100,1000$. The thicker blue line is the catenary for $c = 0$. The black line is the truncated circumference when $c\to\infty$, with equation $x^2+y^2=13.25=\sfrac{1}{E_\infty^2}$.}
\label{fig:numerica}
\end{figure}

In \Cref{fig:3d}, we show a portion of the revolution surface for different values of the parameter $c$. In the interior we have the catenoid in yellow, then the surfaces are ordered as $c$ increases: they are convex for $c<r^2$, the solution is the cylinder for $c=r^2$, it has a concave profile for $c>r^2$ and finally the region of sphere generated by $\rho_\infty$ in blue.

\begin{figure}[htbp]
    \centering
    \includegraphics[width=0.5\textwidth]{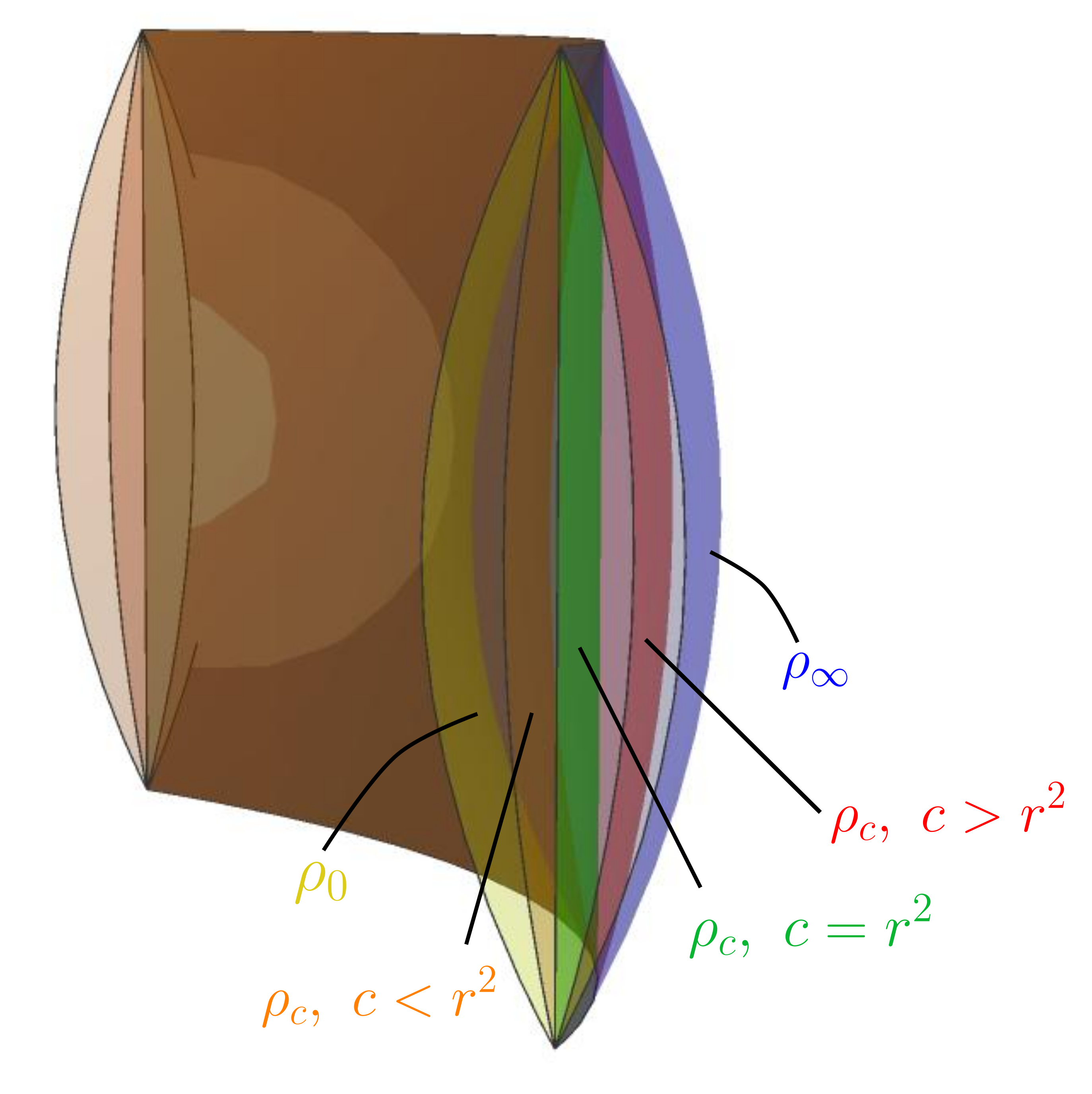}
    \caption{A slice of revolution surfaces obtained from profiles that are minimizers of the energy functional $\mathcal{F}_c$. We choose $h=1$ and boundary conditions $\rho_c(1) = \rho_c(-1) = \sfrac{7}{2}$. From the yellow surface to the blue one, we have the catenoid (yellow) for $c = 0$, the convex surface (orange) for $c = 5$, the cylinder (green) for $ c= 12.25 = r^2$, the concave profile (red) for $c = 50$  and finally the region of a sphere generated by $\rho_\infty$ in blue.}
    \label{fig:3d}
\end{figure}

In \Cref{cor:bvp}, we prove that the boundary value problem \eqref{bvp_num} admits a unique solution in all the regimes. Precisely, if $c <r^2$, and if \eqref{cond_catenaria} holds true, \Cref{prop:esistenza_minimo_Ec}, \Cref{prop:proprieta_bvp} and \Cref{lemma:unicita_soluzioni_1} ensure that there exists a unique minimizer solution of the Euler-Lagrange equation lying above the stable catenary $\rho_0$.
Unfortunately, once $c$ and the geometric parameters $h$ and $r$ are fixed,  we are not able to determine how many solutions the transcendental equation \eqref{p-} has, implying that we do not know the exact number of critical points for the energy functional $\mathcal{F}_c$ in the convex regime.
Hence, we decide to perform a numerical study. Precisely, we rewrite \eqref{p-} as
\begin{equation}
\label{trascr}
h =\frac{E}{2}\operatorname{arccosh}\frac{2r^2-E^2+2c}{E\sqrt{E^2-4c}}=:H(E).
\end{equation}
We carry out a qualitative graphical analysis by studying the intersection points between horizontal lines $y=h$ and the graphs of the function $y=H(E)$ for different values of $c$. 
Precisely, in \Cref{fig:multiple}, fixing different values of $c$ from $0$ to $0.05$, we numerically solve \eqref{trascr} determining the exact number of solutions of \eqref{bvp_num} with boundary conditions $\rho(\pm h) = r$.

Numerical simulations in the case $c< r^2$ (the convex regime) suggest that for a fixed value of $h$ and $c$, there can be one, two or three solutions $E_i$ with $i = 1,2,3$ for \eqref{trascr}, 
see for instance the intersection of the dashed horizontal line $h = 0.4$ with the red curve in \Cref{fig:multiple} corresponding to $c = 0.002$ for the case of three solutions, the intersection of the horizontal line $h = \omega$ with the pink curve in \Cref{fig:multiple} corresponding to $c = 0.0005$ for the case of two solutions or the dashed horizontal line $h = 0.4$ with the top yellow curve in \Cref{fig:multiple} corresponding to $c = 0.05$ for the case of one solution.

\begin{figure}[htbp]
   \centering
   \includegraphics[width=0.5\textwidth]{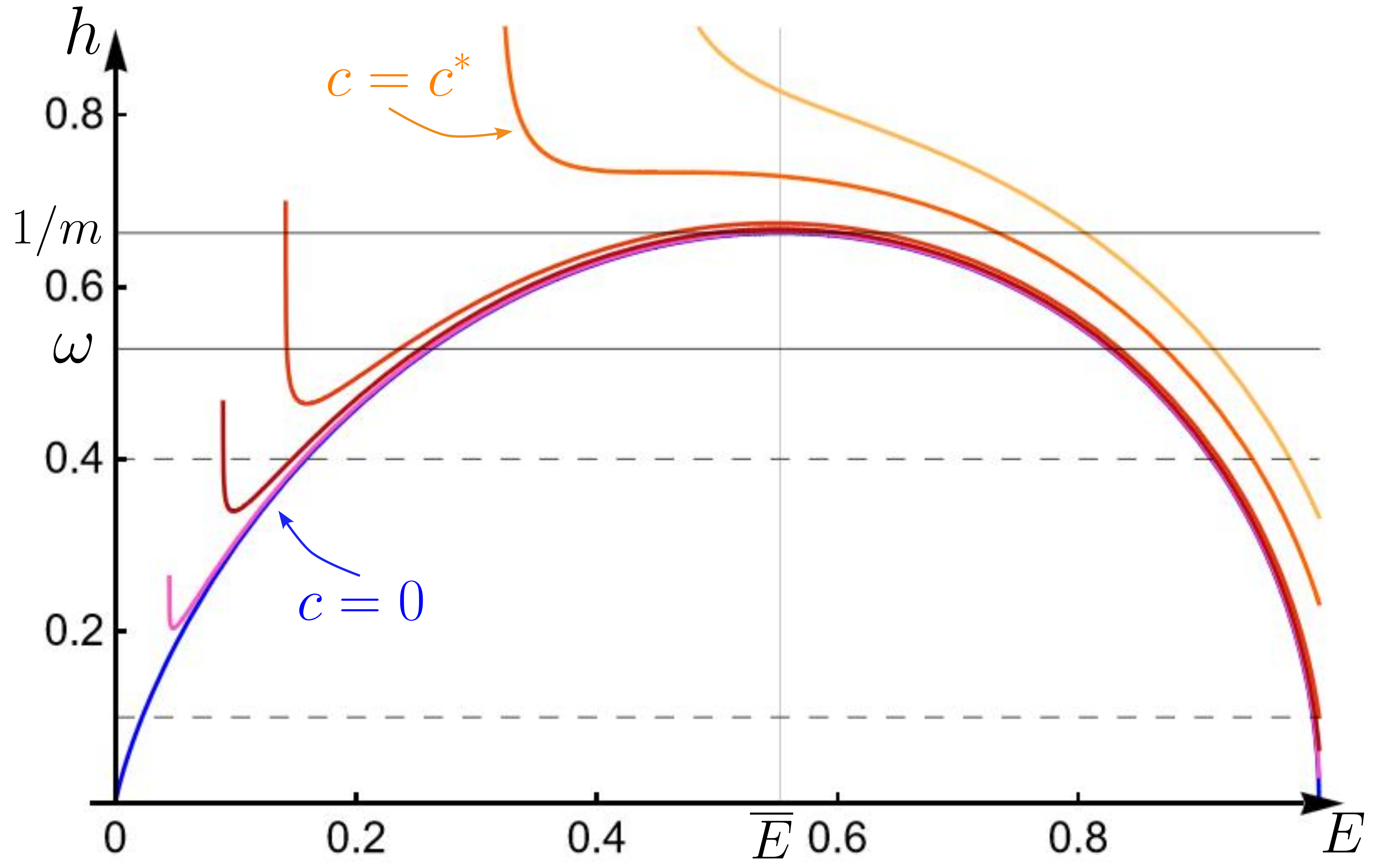}
   \caption{Qualitative study of the number of solutions of \eqref{trascr} for different values of $c$.
    From the bottom to the top, (fixing $r=1$) we have $h=H(E)=0.1, 0.4, \omega\sim 0.528, \sfrac{1}{m}\sim 0.6627$ and $c=0,0.0005,0.002,
 0.005, c^*\sim 0.0257, 0.05$. 
 $\overline{E}$ is the abscissa of the maximum of $H(E)$ for $c=0$. An horizontal line determines two intersections with $H(E)$ in the case $c=0$, three intersections when $c\in(0,c^*)$. If $c\geq c^*$, there is a unique intersection corresponding to the minimizer of the energy functional $\mathcal{F}_c$.}
   \label{fig:multiple}
\end{figure}
From \Cref{fig:multiple}, in the case $c = 0$ and in the regime $0< h\leq \sfrac{1}{m}$, there are always two intersections: they correspond to the two catenaries solutions of the Plateau problem spanning two coaxial rings, one is the stable catenary and the other one is the unstable one.
In particular, the stable catenary is the unique minimizer when $\sfrac{h}{r}<\omega$,
for details we refer to \cite[Remark 2.4]{bllm}. 
Moreover, we also notice that $H'(\overline{E}, c=0)=:H'_0(\overline{E})= 0$ in $\overline{E} =0.552434$ and this maximum point is such that $H_0(\overline{E}) = \sfrac{1}{m}$, i.e the last value  of $\sfrac{h}{r}$ at which catenaries are solutions.
Next, for $c$ small enough and for all $h\leq \omega$, we have three solutions: two intersections are such that $E_i < \overline{E}$ with $i = 1,2$, while the third one has abscissa bigger than $\overline{E}$ and precisely such an intersection 
corresponds to the unique minimizer of $\mathcal{F}_c$ since it lies above the catenary. As expected and theoretically proved, we also notice that the abscissa of this third intersection increases with $c$. Thus, the two intersections that determine $E_i$ with $i=1,2$ correspond to unstable solutions that lie  below the catenary, and thus they cannot be minimizers. 

We finally spot that there are multiple solutions to $\eqref{trascr}$ for $0<c<c^*$ with $c^*\sim 0.0257$. Indeed, repeating the same argument carried out in \Cref{lemma:unicita_soluzioni_1}, we obtain a unique admissible solution of $ H''(E)=0$ equal to
    $$
    E^*(c) = \sqrt{\frac{c ^4-3 c ^3-9 c ^2-\sqrt{c ^8+18 c ^7+15 c ^6-36
   c ^5-33 c ^4+18 c ^3+17 c ^2}-5 c }{3 c ^2-6 c -1}}.
    $$
    Substituting the above expression into $H'(E)$ and imposing that the first derivative vanishes, we numerically solve $H'(E) = 0$ with respect to $c$ and we get (using Mathematica as numerical software) that there exists a unique zero at $c=c^*\sim 0.0257224$. Furthermore, we notice that $H'''(E)\neq 0$ for $E=E^*(c^*)$, implying that $E^*(c^*)$ is an inflection point with horizontal tangent. Finally, if  $c\geq c^*$ there exists a unique $E = E(c)$ solving $\eqref{trascr}$.

\begin{appendix}
    \section{Derivation of $\mathcal{F}_c$}
\label{app:derivazione}
    
In this Appendix, we derive the expression of the energy functional $\mathcal{F}_c$ starting from 
\begin{equation}\label{funz_app}
    \mathcal{E}(\n, S)=\bigintsss_S \left(\gamma+\frac{\kappa}{2}|\nabla_S \n|^2\right)\,dA=
    \bigintsss_S \left(\gamma+\frac{\kappa}{2}\left(|\tens D \n|^2+|\tens L \n|^2\right)\right)\,dA
\end{equation}
and specializing \eqref{funz_app} in the setting of a revolution surface $S$ spanning two coaxial rings of radius $r$ placed at distance $2h$.
Precisely, since $\nabla_S\n=\tens D \n+\boldsymbol{\nu}\otimes \tens L \n$, where $\boldsymbol{\nu}$ is the normal vector to $S$, we immediately end up with
\begin{equation}
\label{e:exp_gradiente_superficiale}
    \abs{\nabla_S \n}^2 = \abs{\tens{D}\n}^2+ \abs{\tens{L}\n}^2,
\end{equation}
where $\tens D \n$ is the covariant derivative of $\n$, while $\tens L$  is the extrinsic curvature tensor.

To compute these terms we proceed by parametrizing the revolution surface $S$. Let 
\[
\vect{Z}:\left\{
\begin{aligned}
    [0,2\pi] \times [-h,h] &\to \R\\
(\varphi,x)&\mapsto \vect{Z}(\varphi,x):= (\rho(x) \cos \varphi,\rho(x)\sin \varphi,x)
\end{aligned}
\right.
\]
where $\rho:[-h,h]\to (0,+\infty)$ is a smooth function satisfying the boundary conditions $\rho(\pm h)=r$. In this setting, we have 
\[
\vect{Z}_\varphi=(-\rho\sin \varphi, \rho \cos\varphi,0), \qquad \vect{Z}_x=(\rho'\cos \varphi, \rho' \sin\varphi,1).
\]
Then, 
\[
dA=|\vect Z_\varphi \times \vect Z_x|d\varphi dx=\rho\sqrt{1+(\rho')^2}\,d\varphi dx.
\]
Since  
\[
\vect{e}_1=\frac{\vect{Z}_\varphi}{|\vect{Z}_\varphi|}, \qquad 
\vect{e}_2=\frac{\vect{Z}_x}{|\vect{Z}_x|}
\]
form an orthonormal basis on the tangent plane to $S$ and since the nematic director $\n$ is a unit vector constrained to be tangent to the surface, we can rewrite $\n$ as 
$$\n=\cos \alpha \vect{e}_1+\sin \alpha \vect{e}_2,$$
where $\alpha=\alpha(x,\varphi) \in [0, \pi]$ is the positive oriented angle formed by the vector $\n$ with $\vect e_1$, which is the direction of parallels.
Since $|\vect Z_\varphi|=\rho$ and $|\vect Z_x|=\sqrt{1+(\rho')^2}$, we explicitly have 
$$
\n=\frac{\cos \alpha}{\rho}\vect{Z}_\varphi+\frac{\sin \alpha}{\sqrt{1+(\rho')^2}}\vect{Z}_x.
$$
The first term $|\tens D\n|^2$ in \eqref{e:exp_gradiente_superficiale} was already computed in \cite[Appendix A]{bllm} and it is given by
\[
|\tens D \n|^2=\frac{\alpha_x^2}{1 + (\rho')^2} + \frac{\alpha_\varphi^2}{\rho^2}-\frac{2\rho'\alpha_\varphi}{\rho^2 \sqrt{1 + (\rho')^2}} +\frac{(\rho')^2}{\rho^2 (1 +(\rho')^2)}.
\]

To compute $\tens L \n$, we proceed as in \cite[Appendix: Geometric preliminaries]{napoli2010equilibrium}. 
The extrinsic curvature tensor $\tens L$ satisfies $\tens L=c_1 \boldsymbol{e}_1\otimes \boldsymbol{e}_1+c_2 \boldsymbol{e}_2\otimes \boldsymbol{e}_2$ where $c_1$ and $c_2$ are the principal curvatures of $S$. Since $S$ is assumed to be a revolution surface, we can apply Meusnier's Theorem to compute $c_1$.
Precisely, we have 
$$c_1=c_P\boldsymbol{\nu}\cdot \boldsymbol{\nu}_P$$
where $c_P$ is the curvature of a parallel on $S$, $\boldsymbol{\nu}$ is the normal vector to the surface and $\boldsymbol{\nu}_P$ is the projection of $\boldsymbol{\nu}$ on the plane of the parallel (i.e. $x=\textup{constant}$). Since parallels of $S$ are circles of radius $\rho$, we have $c_P=\sfrac{1}{\rho}$.
Concerning the normal $\boldsymbol{\nu}$, we can compute it as
$$\boldsymbol{\nu}=\frac{\vect Z_\varphi \times \vect Z_x}{|\vect Z_\varphi \times \vect Z_x|}=\left(\frac{\cos \varphi}{\sqrt{1+(\rho')^2}},\frac{\sin\varphi}{\sqrt{1+(\rho')^2}},-\frac{\rho'}{\sqrt{1+(\rho')^2}}\right).$$
Finally, since $\boldsymbol{\nu}_P$ is its projection on the plane of the parallel, we have
$$
\boldsymbol{\nu}_P=\left(\frac{\cos \varphi}{\sqrt{1+(\rho')^2}},\frac{\sin\varphi}{\sqrt{1+(\rho')^2}},0\right).
$$
Thus, $$c_1=\frac{1}{\rho}\left(\frac{\cos \varphi}{\sqrt{1+(\rho')^2}},\frac{\sin\varphi}{\sqrt{1+(\rho')^2}},-\frac{\rho'}{\sqrt{1+(\rho')^2}}\right)\cdot \left(\frac{\cos \varphi}{\sqrt{1+(\rho')^2}},\frac{\sin\varphi}{\sqrt{1+(\rho')^2}},0\right)=-\frac{1}{\rho\sqrt{1+(\rho')^2}}.$$
The curvature of the meridian, along the direction of $\boldsymbol{e}_2$, is the curvature of a planar curve, that is, see \cite[Chapter 1]{do2016differential}, 
$$c_2=\frac{\rho''}{\left(1+(\rho')^2\right)^{\sfrac{3}{2}}}.$$
Collecting everything, we end up with
$$
\tens L\n=\begin{pmatrix}
    c_1
& 0\\
0 & c_2\end{pmatrix} \begin{pmatrix}
   \n_1\\
\n_2\end{pmatrix}=\begin{pmatrix}
 - \frac{1}{\rho\sqrt{1+(\rho')^2}}  & 0\\
0 & \frac{\rho''}{(1+(\rho')^2)^\frac{3}{2}}
\end{pmatrix}\begin{pmatrix}
    \cos \alpha\\
    \sin \alpha
\end{pmatrix}=\begin{pmatrix}
   - \frac{1}{\rho\sqrt{1+(\rho')^2}}\cos \alpha\\
    \frac{\rho''}{(1+(\rho')^2)^{\sfrac{3}{2}}}
     \sin \alpha
\end{pmatrix}
$$
and in particular 
$$|\tens L\n|^2=c_1^2\cos^2\alpha+c_2^2\sin^2\alpha
=\frac{\cos^2\alpha}{\rho^2\left(1+(\rho')^2\right)}+\frac{(\rho'')^2 \sin^2\alpha}{(1+(\rho')^2)^3}\,.$$
Therefore, by Fubini's Theorem, the energy functional
\eqref{funz_app} reads as
\[
\begin{aligned}
 \mathcal{E}(\n, S)&=\bigintsss_{-h}^h\bigintsss_0^{2\pi}d\varphi dx\left(\gamma\rho\sqrt{1+(\rho')^2}+  \frac{\kappa}{2}\rho\sqrt{1+(\rho')^2} \right.\\
 &\left.\times \left(\frac{\alpha_x^2}{1 + (\rho')^2} + \frac{\alpha_\varphi^2}{\rho^2}-\frac{2\rho'\alpha_\varphi}{\rho^2 \sqrt{1 + (\rho')^2}} +\frac{(\rho')^2}{\rho^2 (1 +(\rho')^2)}
 +\frac{\cos^2\alpha}{\rho^2\left(1+(\rho')^2\right)}+\frac{(\rho'')^2 \sin^2\alpha}{(1+(\rho')^2)^{3}}\right)\right)\\
 &=I_1+I_2+I_3+I_4+ I_5
 \end{aligned}
 \]
 where 
 \begin{equation*}
    \begin{aligned}
 I_1&=2 \pi\gamma\bigintsss_{-h}^h\left(\rho\sqrt{1+(\rho')^2}+\frac{\kappa}{2\gamma}\frac{(\rho')^2}{\rho \sqrt{1 +(\rho')^2}}\right)\,dx,\\
 I_2&=\frac{\kappa}{2}\bigintsss_{-h}^h\bigintsss_{0}^{2 \pi}
 \left(\frac{\cos^2\alpha}{\rho\sqrt{1+(\rho')^2}}+\frac{\rho(\rho'')^2 \sin^2\alpha}{(1+(\rho')^2)^{\sfrac{5}{2}}}\right)\,dxd\varphi, 
\end{aligned}
\end{equation*}
and
\begin{equation*}
    \begin{aligned}
 I_3&=\frac{\kappa}{2}\bigintsss_{-h}^h\frac{\rho}{\sqrt{1 + (\rho')^2}}\bigintsss_{0}^{2 \pi}\alpha_x^2\,d\varphi dx,\\ 
    I_4&=\frac{\kappa}{2}\bigintsss_{-h}^h\frac{\sqrt{1+(\rho')^2}}{\rho}\bigintsss_{0}^{2 \pi}\alpha_\varphi^2\,d\varphi dx, \\
    I_5&=-\frac{\kappa}{2}\bigintsss_{-h}^h\frac{2\rho'}{\rho}\bigintsss_{0}^{2 \pi}\alpha_\varphi\,d\varphi dx. 
\end{aligned}
\end{equation*}
First of all, $I_5=0$: $\n$ is smooth, hence $\alpha(x,0)=\alpha(x,2\pi)$ for any $x\in [-h,h]$. Moreover, in this paper, we choose $\alpha=0$, so that the nematic director $\n$ is aligned along parallels. Then in $\mathcal{E}(\n, S)$, we immediately have that $I_3$ and $I_4$ vanish.
Moreover, $I_2$ simplifies since $\sin \alpha = 0$ while $\cos \alpha=1$ getting that the energy functional is geometric: it depends just on the profile function $\rho$ and up to $2\pi\gamma$, it reduces to 
$$
\begin{aligned}
 \mathcal{E}(\rho)&=\bigintsss_{-h}^h\left(\rho\sqrt{1+(\rho')^2}+c\left(\frac{(\rho')^2}{\rho \sqrt{1 +(\rho')^2}}+\frac{1}{\rho\sqrt{1+(\rho')^2}}\right)\right)\,dx\\
 &=\bigintsss_{-h}^h\sqrt{1+(\rho')^2}\left(\rho+\frac{c}{\rho}\right)\,dx=:\mathcal{F}_c(\rho),   
\end{aligned}
$$ 
having defined $c := \frac{\kappa}{2 \gamma}$.

\end{appendix}

\bigskip
\bigskip

\section*{Acknowledgements}
The authors thank Paolo Tilli for his advices in the proof of \Cref{tilli2} and \Cref{tilli1} and they thank Gaetano Napoli and Luigi Vergori for having put the problem to their attention.\\
GB acknowledges the MIUR Excellence Department Project awarded to the Department of Mathematics, University of Pisa, CUP I57G22000700001. 
CL and AM acknowledge Gruppo Nazionale per la Fisica Matematica (GNFM) of Istituto Nazionale di Alta Matematica (INdAM).

\section*{Funding}
\begin{itemize}
\item GB is supported by European Research Council (ERC), under the European Union's Horizon 2020 research and innovation program, through the project ERC VAREG - {\em Variational approach to the regularity of the free boundaries} (grant agreement No. 853404).
\item GB and LL are supported by Gruppo Nazionale per l'Analisi Matematica, la Probabilit\`a e le loro Applicazioni (GNAMPA) of Istituto Nazionale di Alta Matematica (INdAM) through the INdAM-GNAMPA project 2025 CUP E5324001950001.
\item CL is supported by the MICS (Made in Italy – Circular and Sustainable) Extended Partnership and received funding from the European Union Next-Generation EU (PIANO NAZIONALE DI RIPRESA E RESILIENZA (PNRR) – MISSIONE 4 COMPONENTE 2, INVESTIMENTO 1.3 – D.D. 1551.11-10-2022, PE00000004) and by the Italian Ministry of Research, under the complementary actions to the NRRP “D34Health - Digital Driven Diagnostics, prognostics and therapeutics for sustainable Health care” Grant (\# PNC0000001). 
\end{itemize}

\section*{Conflict of interest and data availability statements}
The authors declare that they have no known competing financial interests or personal relationships that could have appeared to influence the work reported in this paper. Moreover,  no additional data are included.

\printbibliography
\end{document}